%% file: RELTRv2.tex
\documentclass[reqno]{amsart}
\input{headpap}

\usepackage[foot]{amsaddr}

\title{Trace maps on rigid Stein spaces}
\author{Milan Malčić}     
	\address{Mathematisches Institut \\ Im Neuenheimer Feld 205 \\ D-69120 Heidelberg}
		\email{mmalcic@mathi.uni-heidelberg.de} 

\SetSymbolFont{stmry}{bold}{U}{stmry}{m}{n}
\usepackage{bm}
\usepackage{scalerel}

\begin{document} 

\begin{abstract} 
	{We provide a \textit{relative} version of the trace map from the work of Beyer, which can be associated to any finite étale morphism $ X \nach Y$ of} smooth rigid Stein spaces and which then relates  the Serre duality on $X$ with the Serre duality on $Y$.
	Furthermore, we consider the behaviour of any rigid Stein space under (completed) base change to any complete extension field and deduce a commutative diagram relating Serre duality over the base field with the Serre duality over the extension field.
\end{abstract}

\maketitle 

\tableofcontents

\section*{Introduction}

 Serre duality for rigid
Stein spaces was first established in \cite{Chi90}, but  Peter Beyer subsequently provided another proof appearing in his dissertation \cite[Satz 7.1]{Bey97b} and outlined it in  \cite[Remark 5.1.5]{Bey97a}. Given a smooth $d$-dimensional rigid Stein space $X$ over a non-archimedean field $K$, Beyer's approach allows for an explicit description of the trace map $$t_X \colon H^d_c(X, \omega_X) \nach K$$ that turns the canonical bundle $\omega_X$ into a dualizing sheaf. Here $H^*_c(X,-)$ denotes the compactly supported cohomology of the rigid space $X$ with coefficients in coherent sheaves, defined as in \cite[\S 1.1]{Bey97a} and \cite[\S1]{Chi90} and topologized as in \cite[§1.3]{Bey97a} and \cite[1.6]{vdP92}.

Our first goal is to provide a relative version of the trace map for rigid Stein spaces. In fact, 
 if $\alpha \colon X \nach Y$ is a finite étale morphism of smooth {connected} $d$-dimensional Stein spaces over $K$, there is a  natural candidate
$$
t_\alpha \colon \alpha_*\omega_X \nach \omega_Y
$$ 
that appears in \cite[p.\ 48]{SV23} and is already called \quot{the relative trace map} there. Locally, for an affinoid $\Sp(A)\subseteq Y$ and its preimage $\Sp(B) \subseteq X$, $t_\alpha$ is defined by
$$
\begin{aligned}
	\Omega^d_{B/K} = \Omega^d_{A/K} \otimes_A B & \nach \Omega^d_{A/K} \\
	\omega \otimes b & \auf \Tr_{B/A}(b) \cdot \omega,
\end{aligned}
$$
where  $\Omega^d_{A/K}:=\wedge^d \Omega^1_{A/K}$ and $\Omega^1_{A/K}$ denotes the  \quot{universally finite} differential module  (characterized by being finitely generated over $A$ and universal for derivations into finitely generated $A$-modules), and $\Tr_{B/A}$ is the trace of the finitely generated projective $A$-module $B$. We prove that $t_\alpha$ indeed deserves to be called the relative trace map:

\begin{customthm}{1}[Theorem \ref{simpl thm rel tr}, Proposition \ref{prop compatib}]\label{introthm1}
	Let $\alpha \colon X \nach Y$ be a finite étale morphism of smooth {connected} $d$-dimensional Stein spaces over $K$. 
	Then the following  diagram commutes:
	\[
	\begin{tikzcd}
		& {H^d_c(Y, \omega_Y)} \arrow[rd, "t_Y"]  &    \\
		{H^d_c(Y, \alpha_*\omega_X)} \arrow[ru, "{H^d_c(Y, t_\alpha)}"] \arrow[rd, "\cong"'] &                                          & K. \\
		& {H^d_c(X, \omega_X)} \arrow[ru, "t_X"'] &   
	\end{tikzcd}
	\]
	In particular, letting $\Ext^i_X(\alpha^* \mc G, \omega_X) \nach \Ext^i_Y(\mc G, \omega_Y)$ denote the morphism constructed from $t_\alpha$ and the adjunction morphism $\mc G \nach \alpha_* \alpha^* \mc G$ as in {Proposition} \ref{prop compatib}, one obtains a commutative diagram of Serre duality pairings
	\[
	\begin{tikzcd}
		{H^{d-i}_c(X, \mc \alpha^* \mc G)^\vee} \arrow[r, "\sim"] \arrow[d] & {\Ext^i_X(\alpha^* \mc G, \omega_X)} \arrow[d]                                                                                                   \\
		{H^{d-i}_c(Y, \mc G)^\vee} \arrow[r, "\sim"]                        & {\Ext^i_Y(\mc G, \omega_Y)}                                                                                                                                       
	\end{tikzcd}
	\]
	for every coherent sheaf $\mc G$ on $Y$ and all $i \geq 0$.
\end{customthm}

This theorem is of interest in the study of reciprocity laws for {$(\varphi_L, \Gamma_L)$}-modules over Lubin-Tate extensions as done by  Schneider and Venjakob. Indeed, in  \cite[\S 4.2]{SV23}, Theorem \ref{introthm1} is used as a conceptual way to obtain functorial properties of pairings arising from Serre duality on certain specific rigid Stein spaces. We mention that the recent work of Abe and Lazda \cite{AL20} constructs a trace map on proper pushforwards of analytic adic spaces, and that some of their results can be related to our Theorem \ref{introthm1}. Whereas \cite{AL20} develops a general machinery, our constructions are rather explicit, and it would be an interesting project to fully compare these two approaches.

The proof of Theorem \ref{introthm1} builds on 
the technique of investigating $t_X$ in the special case of so-called \emph{special affinoid wide-open
	spaces} (which we review in \S \ref{subsec sp wid op} below) via a relation between algebraic local cohomology and compactly supported rigid cohomology established in  \cite{Bey97a}. For this purpose we also prove a generalized version of Bosch’s theorem \cite[Satz 7.1]{Bos77} on the
connectedness of formal fibers, which closes an argumentative gap in \cite{Bey97a}. 
Namely,  there is a crucial technical lemma \cite[Lemma 4.2.2]{Bey97a} underlying Beyer's arguments, which asserts that special affinoid wide-open spaces can be characterized in two equivalent ways, and the proof of this lemma relies on Bosch's connectedness result. However, Bosch's result contains the assumption that the affinoid algebra under consideration is \emph{distinguished} (see Definition \ref{defi disting}) and this assumption is not satisfied in the general setting of \cite{Bey97a}, resulting in an argumentative gap. 

 The generalized connectedness result is proved in Theorem \ref{thm faser zshgd im allg} for affinoid algebras that become distinguished after a base change to a finite Galois extension. 
This indeed bridges the gap in the setting of smooth special affinoid wide-open spaces, since Lemma \ref{Erw nach der wir disting sind} proves that the relevant affinoid algebras can be made distinguished after a base change to a finite Galois extension. We explain this further in \S \ref{subsec sp wid op}.

Our second goal is to show that Beyer's trace map  behaves well under (completed) base change to any complete extension field $K'/K$. 
More precisely, for any separated rigid space $X$ over $K$, let
$
X':=X \ctens{K} K'
$
denote the base change of $X$ to $K'$ and let $
\mc F \rightsquigarrow \mc F'
$ be the exact \quot{pullback} functor
from coherent sheaves on $X$ to coherent sheaves on $X'$.
For a Stein space $X$ over $K$, we construct comparison maps
$
H^j_{c}(X, \mc F) \tens{K} K' \nach {H^j_{c}(X', \mc F')}
$
in Section \ref{base change section} and prove:
	\begin{customthm}{2}[Corollary \ref{bc ult result}]\label{intro thm bc}
		Let $X$ be a smooth rigid Stein $K$-space of dimension $d$. Then,  for every coherent sheaf $\mc F$ on $X$, the diagram
	%	\begin{equation}\label{diag intro base change}\tag{$\ast$}
		%	\begin{gathered}
			\[
				\begin{tikzcd}
					{H^{d-i}_c(X', \mc F')} \arrow[r, "\times", phantom]         & {\Ext^{i}_{X'}(\mc F', \omega_{X'})} \arrow[r]    & H^{d}_c(X', \omega_{X'})   \arrow[r, "t_{X'}"] & K'          \\
					{H^{d-i}_c(X, \mc F)} \arrow[r, "\times", phantom] \arrow[u] & {\Ext^{i}_X(\mc F, \omega_X)} \arrow[u] \arrow[r] & H^{d}_c(X, \omega_{X}) \arrow[u]  \arrow[r, "t_X"] & K \arrow[u]
				\end{tikzcd}
			\]
%			\end{gathered}
	%	\end{equation}
		commutes for all $i \geq 0$.
	\end{customthm}
\noindent This result is likewise relevant in the context of  \cite[\S 4.2]{SV23}, where it is used to argue that the constructions carried out there are compatible with base change. 

\subsection*{Acknowledgements}

This article is based on the author's Ph.D.\ thesis \cite{Mal23} written under the supervision of Otmar Venjakob and we thank him for his advice and many helpful discussions. We thank Katharina Hübner for her valuable remarks on \cite{Mal23} and for her careful reading of this manuscript. We are also grateful to Peter Schneider, who showed interest in the genesis of \cite{Mal23} and was helpful with several questions. Finally, Section \ref{sec bosch} has resulted from an email correspondence with Werner Lütkebohmert and we thank him for his valuable input. This research was funded by the Deutsche Forschungsgemeinschaft (DFG, German Research Foundation) TRR 326 \textit{Geometry and Arithmetic of Uniformized Structures}, project number 444845124.	

\subsection*{Data availibility statement}
Data sharing not applicable to this article as no datasets were
generated or analysed during the current study.

\subsection*{Conflict of interest statement}
The author states that there is no conflict of interest.

\subsection*{Notation and conventions}

We use the language of \cite{BGR84}, i.e.\ we work with rigid spaces in the sense of Tate.
 All rigid spaces are assumed to be separated. (Note that this assumption is superflous for Stein spaces, since they are automatically separated, see \cite[Satz 3.4]{Lut73}.)

Recall that, if $\{X_i\}_{i \in \Sigma}$ is a subset of the set of connected components of a rigid space $X$, then $\cup_{i \in \Sigma} X_i$ is admissible open and $\{X_i\}_{i \in \Sigma}$ is an admissible covering thereof. Moreover, $\{X_i\}_{i \in \Sigma}$ are precisely the connected components of $\cup_{i \in \Sigma} X_i$.  \label{cf op}
 The admissible opens of $X$ that arise in this way are precisely the admissible opens of $X$ that are also analytic sets, i.e.\ \quot{open and closed}.
 We refer to such subsets as \emph{clopen}.

Let $K$ be a complete non-archimedean nontrivially valued and perfect field, $\mf o_K$ be the ring of integers in $K$ and
 $k$ be the residue field of $K$.
Let $K\langle \xi_1, \ldots, \xi_n \rangle$ be the free Tate algebra in $n$ variables, for which we also write $T_n$ or $T_n(K)$ when we wish to emphasize the base field. We set 
	$\mb D^n:=\mb D^n_K:= \Sp K\langle \xi_1, \ldots, \xi_n \rangle$ and
	$ \mathring{\mb D}^n:=\mathring{\mb D}^n_K:=\{x \in \mb D^n \colon \btrg{\xi_i(x)}<1 \tn{ for all } i=1,\ldots, n\}$.
The assumption that K is perfect, originating from \cite{Bey97a}, is required solely for ensuring the existence of the residue map on local cohomology \cite[Definition 4.2.7]{Bey97a} for any given point. In \S\ref{subsec: res map}, we discuss this residue map, which serves as a useful tool in our analysis.

Recall that the reduction of an affinoid algebra $R$ is $\widetilde R=\oo R/ \check R$ where 
$\oo R =\{f \in R \colon \btrg{f}_{\ssup} \leq 1\}$ is the $\mf o_K$-algebra of all power-bounded elements in $R$
and $	\check R = \{f \in R \colon \btrg{f}_{\ssup} < 1\}$ is the $\oo R$-ideal of all topologically nilpotent elements in $R$. 	The \emph{reduction} $\widetilde{Z}$ of the affinoid space $Z:=\Sp(R)$ is the affine algebraic variety given by the maximal spectrum of $\widetilde{R}$.
There is a functorial reduction map $p \colon Z \nach \widetilde{Z}$ defined as in \cite[\S 7.1.5]{BGR84}, which is surjective by \cite[7.1.5/Theorem 4]{BGR84}. For the reduction of a point $z \in Z$, we often use the notation
	$p(z)$  and $\wt z$
interchangeably. 	For $z \in Z$, the fibers 
	$Z_+(z):=p^{-1}(p(z))$
of the reduction map $p \colon Z \nach \wt Z$ are called \emph{the formal fibers of $Z$}, the terminology and notation being as in \cite{Bos77}.

All rings are assumed commutative and unital. We recall the following characterization of étale ring maps from a field: For any field $F$ and any ring $A$, a ring map $F \nach A$ is étale if and only if $A$ is isomorphic as an $F$-algebra to a finite product of finite separable field extensions of $F$, see \citestacks{00U3}.
A torsion-free morphism $A \nach B$ from an integral domain $A$ into a ring $B$ is called \emph{separable} if the induced map $Q(A)\nachinj Q(B)$ of the fraction field $Q(A)$ of $A$ into the total ring of fractions $Q(B)$ of $B$ is étale. 
Accordingly, a morphism $\Sp(S) \nach \Sp(R)$ of affinoid spaces is called separable if the ring morphism $R \nach S$ is separable.

If $A$ is a ring and $\mf a$ an ideal of $A$, we write $\com{A}{\mf a}$ for the $\mf a$-adic completion of $A$. If $A$ is local, we write $\wh A$ for the completion with respect to the maximal ideal. If $M$ is an $A$-module, we let $H^j_\mf{a}(M)$ denote the $j$-th local cohomology of $M$ with support in $\mf a$ (see, for instance, \cite{Hun07}). It is computed by deriving the left-exact functor $\Gamma_{\mf a}(M):=\{m \in M \colon \mf a^tM=0 \tn{ for some } t \in \N\}$.
\section{A generalized version of a connectedness result by Bosch}\label{sec bosch}
\subsection{Distinguished affinoid algebras}
We recall the following notion from \cite[6.4.3/Definition 2]{BGR84}:
\begin{definition}\label{defi disting}
	Let $R$ be a $K$-affinoid algebra. A surjective morphism $\alpha \colon K \langle \xi_1, \ldots \xi_m \rangle \nachsurj R$ is called \emph{distinguished} if the residue norm $|\cdot|_\alpha$ coincides with the supremum semi-norm $|\cdot|_{\ssup}$ on $R$. A $K$-affinoid algebra $R$ is called \emph{distinguished} if for some $m \geq 0$ it admits a distinguished $\alpha  \colon K \langle \xi_1, \ldots \xi_m \rangle \nachsurj R$.
\end{definition}
\begin{definition}[\hspace{-0.2mm}{\cite[p.\ 138]{Bos70}}]\label{def abs reduced}
	A $K$-affinoid algebra $R$ is called \emph{absolutely reduced} over $K$, if for every complete field extension $K'$ of $K$, the algebra $R \ctens{K} K'$ is reduced.
\end{definition}
\begin{lemma}\label{Erw nach der wir disting sind}
	Let $R$ be an absolutely reduced $K$-affinoid algebra. Then there exists a finite Galois field extension $K'/K$ such that the algebra $R \tens{K} K'$ is distinguished.
\end{lemma}
\begin{proof}
	In  the proof of \cite[Lemma 2.7]{Bos71}, Bosch constructs a finite field extension $K_1/K$ such that the algebra $R \tens{K} K_1$ is distinguished. (Note that $R \tens{K} K_1=R \ctens{K} K_1$ since $K_1$ is finite over $K$.) The field $K_1$ is, in the notation of loc.\ cit., constructed by adjoining certain algebraic elements $c_{i \nu}, i=1, \ldots, n, \nu \in \N_0^m$, only finitely many of which are non-zero, to $K$. The $c_{i \nu}$ are obtained by using the density of the algebraic closure of $K$ in its completion to approximate certain elements $c_{i\nu}'$ of the completion. But, by using the density of the separable algebraic closure of $K$ \cite[3.4.1/Proposition 6]{BGR84}, we can assume that the $c_{i \nu}$ are separable algebraic over $K$ and set $K'$ to be the Galois closure of the finite separable extension $K_1$. Then the rest of the argument in the proof of \cite[Lemma 2.7]{Bos71} can be applied verbatim to show that $R \tens{K} K'$ is distinguished.
\end{proof}
\subsection{Behaviour of formal fibers under base change} 
Recall that, for an affinoid space $Z$ and $z \in Z$, the fibers 
$Z_+(z):=p^{-1}(p(z))$
of the reduction map $p \colon Z \nach \wt Z$ are called \emph{the formal fibers}.
\begin{proposition}\label{faserweise surj}
	Let $R$ be a $K$-affinoid algebra and $K'/K$  a finite Galois extension. Consider $S:=R \tens{K}K'.$ Let $\iota \colon R \nach S$ be the canonical inclusion and 
	$$
	\phi \colon Z'=\Sp(S) \nach Z=\Sp(R)
	$$
	the associated morphism of rigid spaces over $K$, which fits into the commutative diagram
	\begin{equation}\label{faserdiag1}
		\begin{tikzcd}
			Z' \arrow[r, two heads] \arrow[d, "\phi"'] & \wt{Z}' \arrow[d, "\wt{\phi}"] \\
			Z \arrow[r, two heads]                        & \wt Z.    
		\end{tikzcd}
	\end{equation}
	Then the reduction map $Z' \nach \wt{Z}'$ is fiberwise surjective in the sense that for every $z \in Z$, the induced map 
	$$
	\invers{\phi}(z) \nach \invers{\wt{\phi}}(\wt z)
	$$
	is surjective. 
\end{proposition}
\begin{proof}
	The morphism $\phi$ is finite, so $\faserphi z$ is a finite set, say
	$\faserphi z=\{z_1', \ldots, z_n'\}.$
	Then $\wt \phi$ is also finite by \cite[6.3.5/Theorem 1]{BGR84}, so the complement $\faserphit{\wt z} \setminus \{\wt z_1', \ldots, \wt z_n'\}$ is a finite set and we have to show that it is in fact empty. Suppose that it is non-empty, say
	$$\faserphit{\wt z} \setminus \{\wt z_1', \ldots, \wt z_n'\}=\{\wt z_{n+1}', \ldots, \wt z_s'\}$$
	for some elements $\wt z_{n+1}', \ldots, \wt z_s' \in \wt Z'$ with lifts $z_{n+1}', \ldots, z_s' \in Z'$. Based upon this assumption, we will construct an $f \in \oo R$ such that
	$$|f(z)|=1 \quad \tn{ and } \quad |f(\phi(z_{n+1}'))|<1.$$ %\tn{ for } j=n+1, \ldots, s
	But then, by \cite[7.1.5/Proposition 2]{BGR84}, $|f(z)|=1$ means that $\wt f(\wt z) \neq 0$ whereas $|f(\phi(z_{n+1}'))|<1$ means that $\wt f(\wt z) = 0$ since $\phi(z_{n+1}')$ also reduces to $\wt z$. Thus we arrive at a contradiction and the assertion is proved. It remains to construct such an $f \in \oo R$. For this, we first choose a $\wt g \in \wt S$ such that 
	\begin{equation}\label{eig g tilde}
		\wt g(\wt z_j')=1 \tn{ for } j=1, \ldots n \quad \tn{ and } \quad \wt g(\wt z_{n+1}')=0.
	\end{equation}
	This is possible since
	$
	\wt S \nach \prod_{j=1}^{n+1} \wt S/\mf m_{\wt z_j'}
	$
	is surjective by the Chinese Remainder Theorem. Then we choose a lift $g \in \oo S$ of $\wt g$. 
	Note that (\ref{eig g tilde}) then yields 
	\begin{equation}\label{eig g}
		|g( z_j')|=1 \tn{ for } j=1, \ldots n \quad \tn{ and } \quad |g( z_{n+1}')|<1
	\end{equation}
	by \cite[7.1.5/Proposition 2]{BGR84}. Next, the Galois group $G=\Gal(K'/K)$ acts in an obvious way on $R \tens{K} K'=S$ by $R$-algebra homomorphisms, and it is easy to prove that the fixed elements satisfy $S^G=R.$ Therefore,
	$$f:=\prod_{\sigma \in G} \sigma(g) \in R.$$
	Moreover, any morphism $S \nach S$ is contractive with respect to $|\cdot|_{\ssup}$ by  \cite[3.8.1/Lemma 4]{BGR84}, so in particular we have $\supr{\sigma(g)} \leq \supr{g} \leq 1$ for all $\sigma \in G$. Hence
	\begin{equation}\label{ungl galois}
		\supr{f} \leq \prod_{\sigma \in G} \supr{\sigma(g)} \leq 1.
	\end{equation}
	
	Since $\iota \colon R \nach S$ is finite, it is an isometry with respect to $\supr{\cdot}$ by \cite[3.8.1/Lemma 6]{BGR84}, i.e.\ the $\supr{\cdot}$ of $S$ restricts to the $\supr{\cdot}$ of $R$. Thus the inequality (\ref{ungl galois}) shows that  in fact $f \in \oo R.$ Next we claim that 
	\begin{equation}\label{betr gl galois}
		|\sigma(g)(z_1')|=1 \quad \tn{ for all } \sigma \in  G.
	\end{equation}
	To see this, note that each $\sigma \in G$ permutes $\{\mf m_{z_1'}, \ldots, \mf m_{z_n'}\}$ so $\invers \sigma(\mf m_{z_1'})=\mf m_{z_j'}$ for some \linebreak $j \in \{1, \ldots n\}$ and thus $\sigma$ induces an isomorphism 
	$
	\sigma \colon S/\mf m_{z_j'} \isoauf S/\mf m_{z_1'}
	$
	mapping $g \tn{ mod } \mf m_{z_j'}$ to $\sigma(g) \tn{ mod } \mf m_{z_1'}$, which means that $|g(z_j')|=|\sigma(g)(z_1')|$. Since $|g(z_j')|=1$ by (\ref{eig g}), this proves (\ref{betr gl galois}). Finally, we compute
	\begin{align*}
		|f(z)|&=|f(\phi(z_1'))| 
		= |\iota(f)(z_1')| 
		= |\prod_\sigma \sigma(g) (z_1')|
		= \prod_\sigma |\sigma(g) (z_1')|
		=\prod_\sigma 1 
		= 1
	\end{align*}
	and similarly
	$$
	|f(\phi(z_{n+1}'))| = \prod_\sigma |\sigma(g) (z_{n+1}')| 
	= |g(z_{n+1}')|\cdot \prod_{\sigma \neq \id} |\sigma(g) (z_{n+1}')| 
	<1
	$$
	since $|g(z_{n+1}')|<1$ and  $|\sigma(g) (z_{n+1}')| \leq 1$.
\end{proof}
\begin{corollary}\label{faserbeziehung}
	Let $Z$ be a $K$-affinoid space and $K'/K$  a finite Galois extension. Consider the base change  $Z':=Z \tens{K}K'$ of $Z$ to $K'$ and 
	the associated morphism 
	$
	\phi \colon Z' \nach Z
	$ 
	of rigid spaces. Consider any $z \in Z$, and let $\{z_1', \ldots, z_n'\}=\faserphi z$. Then we have the following relation between formal fibers
	\begin{align}\label{fasergl}
		Z_+(z)= \bigcup_{i=1}^n \phi(Z'_+(z_i')).
	\end{align}
\end{corollary}
\begin{proof}
	The inclusion \quot{$\supseteq$} in (\ref{fasergl}) is clear due to the commutativity of the diagram (\ref{faserdiag1}).
	To show the reverse inclusion, let $y \in Z_+(z)$ and take a preimage $y' \in Z'$ under $\phi$. We need to show that $y' \in Z'_+(z_i')$ for some $i$. By the commutativity of (\ref{faserdiag1}) and since $y \in Z_+(z)$, we find that $\wt{\phi}(\wt y')= \wt y = \wt z$, so $\wt y' \in \faserphit{\wt z}$. Now $\faserphit{\wt z}=\{\wt z_1', \ldots, \wt z_n'\}$ by Proposition \ref{faserweise surj}, whence $\wt y'=\wt z_i'$ for some $i$, i.e. $y' \in Z'_+(z_i')$. 
\end{proof}
\subsection{The generalized connectedness result}
If $R$ is a $K$-affinoid algebra, then $R$ is Noetherian and hence contains only finitely many minimal prime ideals, say $\mf p_1, \ldots, \mf p_s$. As in \cite[p.\ 6]{Bos70}, we say that $R$ \emph{has pure dimension} (or \emph{is equidimensional}) if
$
\dim R/\mf p_1 = \dim R/\mf p_2= \ldots = \dim R/ \mf p_s,
$
or, equivalently, 
$
\dim R = \dim R/\mf p_1 = \dim R/\mf p_2= \ldots = \dim R/ \mf p_s.
$
The following Theorem \ref{bosch faser zshgd}, which is due to Bosch, is an important technical result concerning the connectedness of formal fibers:
\begin{theorem}[\!\!\!{\cite[Satz 6.1]{Bos77}}]\label{bosch faser zshgd}
	Let $R$ be a distinguished $K$-affinoid algebra which has pure dimension and let $Z=\Sp(R)$. Then, for every $z \in Z$, the formal fiber $Z_+(z)=\invers{p}(p(z))$ of the reduction map $p \colon Z \nach \wt Z$ is connected.
\end{theorem}
We generalize Bosch's theorem to the case of a not necessarily distinguished affinoid algebra:
\begin{theorem}\label{thm faser zshgd im allg}
	Let $R$ be a $K$-affinoid algebra which has pure dimension and let $Z=\Sp(R)$. Suppose that there exists a finite Galois extension $K'/K$ such that $S:=R \tens{K}K'$ is distinguished. Then, for every $z \in Z$, the formal fiber  $Z_+(z)=\invers{p}(p(z))$ of the reduction map $p \colon Z \nach \wt Z$ is connected.
\end{theorem} 
We postpone the proof for a moment to note that the condition regarding the existence of $K'$ is satisfied whenever $R$ is absolutely reduced, by Lemma \ref{Erw nach der wir disting sind}. In particular, since $R$ is absolutely reduced whenever $Z=\Sp(R)$ is smooth, we deduce:
\begin{corollary}\label{cor faser zshgd}
	Let $R$ be a $K$-affinoid algebra such that $Z=\Sp(R)$ is smooth and connected. Then, for every $z \in Z$, the formal fiber $Z_+(z)=\invers{p}(p(z))$ of the reduction map $p \colon Z \nach \wt Z$ is connected.
\end{corollary}
\begin{proof}[Proof of Theorem \ref{thm faser zshgd im allg}]
	Let $\iota \colon R \nach S$ be the canonical inclusion and  
	$
	\phi \colon Z'=\Sp(S) \nach Z=\Sp(R)
	$
	the associated morphism of rigid spaces, which fits into the commutative diagram
	\begin{equation}\label{diag bw red}
		\begin{tikzcd}
			Z' \arrow[r, two heads] \arrow[d, "\phi"'] & \wt{Z}' \arrow[d, "\wt{\phi}"] \\
			Z \arrow[r, two heads]                        & \wt Z.    
		\end{tikzcd}
	\end{equation}
	The morphism $\iota \colon R \nach S$ is finite  flat and injective, so $\phi$ is finite flat and surjective. Since the morphism $\phi$ is finite, $\faserphi z$ is a finite set, say
	$$\faserphi z=\{z_1', \ldots, z_n'\}.$$
	Since $R$ has pure dimension, the base change $S=R \tens{K} K'$ also has pure dimension by \cite[Lemma 2.5]{Bos70}. Because $S$ is moreover distinguished, the formal fibers $Z'_+(z_i'), i=1, \ldots, n$ are connected by Theorem \ref{bosch faser zshgd}. Recall that formal fibers are  admissible open in the ambient affinoid space (as is explained in \cite[page 26]{Bos77}), so $Z'_+(z_i')$ is open in $Z'$. Being a flat map between quasi-compact rigid $K$-spaces, $\phi$ is open by \cite[Corollary 5.11]{BL93}, so $\phi(Z'_+(z_i'))$ is open in $Z$ and hence a rigid (sub)space. In particular, the restriction $Z'_+(z_i') \nachsurj \phi(Z'_+(z_i'))$ is a surjective map of rigid spaces whose domain is connected, whence the codomain $\phi(Z'_+(z_i'))$ is also connected.
	%	Moreover, because of the commutativity of the diagram (\ref{diag bw red}) above, $\phi$ induces a morphism
	%	$$
	%	Z'_+(z_i') \nach Z_+(z) \quad \tn{ for each } i=1, \ldots,n
	%	$$
	On the other hand, $\phi$ is finite and hence proper, so $\phi$ maps closed analytic subsets to closed analytic subsets by \cite[Satz 4.1 and its proof]{Kie67a}, which is why $\phi(Z'_+(z_i'))$ is clopen in $Z$. Due to the commutativity of the diagram (\ref{diag bw red}) above, $\phi(Z'_+(z_i'))$ is contained in $Z_+(z)$, i.e.\ we can regard it as a (clopen) subset of $Z_+(z)$.  Therefore, being clopen and connected, each $\phi(Z'_+(z_i'))$ is a connected component of $Z_+(z)$. On the other hand, 
	\begin{align}\label{faservereinigung}
		Z_+(z)= \bigcup_{i=1}^n \phi(Z'_+(z_i')) 
	\end{align}
	by Corollary \ref{faserbeziehung}. 
	Next we will show that $\phi(Z'_+(z_i'))=\phi(Z'_+(z_1'))$ for all $i=1, \ldots, n$, which then implies that $Z_+(z)=\phi(Z'_+(z_1'))$ due to (\ref{faservereinigung}) and thus completes the proof of the theorem. Since $\phi(z_i')=z=\phi(z_1')$, we conclude that $z \in \phi(Z'_+(z_1')) \cap \phi(Z'_+(z_i'))$, so $\phi(Z'_+(z_1')) \cap \phi(Z'_+(z_i')) \neq \varnothing$ for all $i$ which means that the connected components $\phi(Z'_+(z_1'))$ and $\phi(Z'_+(z_i'))$ must coincide. 
\end{proof}

\section{Review of Beyer's residue maps and trace maps} \label{sec: gt}

In this section we summarize Beyer's main results and in particular his construction of the trace map yielding Serre duality. He first constructs the trace map in the special case of so-called \emph{special affinoid wide-open spaces} (which we review in Subsection \ref{subsec sp wid op} below), then he uses the fact that a Stein space
can be exhausted by subspaces of this type. To prove that the trace maps in the covering
are compatible and thus glue to a global trace map, he uses a relation to the local
cohomology groups of the underlying affine schemes (which are defined as in, say, \cite[\S 3.1]{Bey97a} or, more generally, in \cite{Hun07}). This relation to local cohomology is reviewed in {Subsection} \ref{subsec relation} below and used to prove the main result of Section \ref{sec: rel trace}, i.e.\ Theorem \ref{introthm1} from the introduction.

\subsection{Special affinoid wide-open spaces}\label{subsec sp wid op}
\begin{definition}[\!\!{\cite[Definition 4.2.1]{Bey97a}}]\label{def spec wide op}
	Let $Z$ be an affinoid space. A subset $\mathring{W}\subseteq Z$ is called a {\emph{special affinoid wide-open space}} if %one of the equivalent conditions in Lemma \tn{\ref{lem spec wide op}} below is satisfied.
	$\mathring W$ is the preimage of a finite set of points under the reduction map $p \colon Z \nach \widetilde{Z}$.
\end{definition}
We note that a special affinoid wide-open space need not be affinoid. For example, $\Dnull^n$ with $n>0$ is a special affinoid wide-open space that isn't affinoid. In fact, one can show that an affinoid space which is special affinoid wide-open is necessarily zero-dimensional. 

There is a gap in the proof of the crucial result \cite[Lemma 4.2.2]{Bey97a} on equivalent characterizations of special affinoid wide-opens: Bosch's theorem on the connectedness of formal fibers  of distinguished affinoid algebras (Theorem \ref{bosch faser zshgd})  is applied to a not necessarily distinguished affinoid algebra. Our generalization of Bosch's Theorem ({Corollary} \ref{cor faser zshgd}) can be used to remedy this gap, as we explain in the proof of Lemma \ref{eine implikation fuer spec wide op} below.
\begin{lemma}\label{eine implikation fuer spec wide op}
	Let $Z=\Sp(R)$ be a smooth and connected affinoid space, $p \colon Z \nach \widetilde{Z}$ the reduction map and let $\mathring{W}\subseteq Z$ be a special affinoid wide-open space. %the preimage of a finite set of points of the reduction $\widetilde{Z}$ under the reduction map $p \colon Z \nach \widetilde{Z}$. 
	Then there exists a finite surjective morphism $\pi \colon Z \nach \mb D^d$ such that $\oo{W}$ is a  union of connected components of $\pi^{-1}(\mathring{\mb D}^d)$. For any such morphism $\pi$, $\pi^{-1}(\mathring{\mb D}^d)$ consists of only finitely many connected components.  \\
	More precisely: For $\oo W=p^{-1}(\{\widetilde{z_1}, \ldots, \widetilde{z_r}\})$, there exists a finite surjective morphism \allowbreak $\wt{\pi} \colon \wt{Z} \allowbreak \nach \mb A^d_{k}$ that maps all the $\widetilde{z_i}$ to $0$, and any lift $\pi$ of any such $\wt \pi$ satisfies the desired properties above. In this setting, if $\invers{\widetilde{\pi}}(0)=\{\widetilde{z_1}, \ldots, \widetilde{z_r}, \widetilde{z}_{r+1}, \ldots, \widetilde{z_s}\}$ is the full fiber over zero, then the connected components of $\pi^{-1}(\Dnull^d)$ are precisely the  $\invers{p}(\widetilde{z_1}), \ldots, \invers{p}(\widetilde{z_s})$ and $\mathring{W}$ is the union of the components $\invers{p}(\widetilde{z_1}), \ldots, \invers{p}(\widetilde{z_r})$.
\end{lemma}
\begin{proof}
	This is the content of \cite[Lemma 4.2.2]{Bey97a}, but we give an outline of the argument to illustrate where our generalized connectedness theorem plays a role.  By \cite[Lemma 1.1.4]{BKKN67}, we can choose {a finite surjective morphism $\pi \colon Z \nach \D^d$ such that $\pi(z_i)=0$ for all $i=1, \ldots, n$}. Taking its reduction, we obtain a $\wt \pi$ as in the assertion.  
	On the other hand, given any $\wt \pi$ as in the assertion,  then any lift $\pi$ of $\wt \pi$ is finite surjective by the results of \cite[\S 6.3]{BGR84}. Moreover, let $\widetilde{z}_{r+1}, \ldots, \widetilde{z_s} \in \widetilde{Z}$ be such that $\invers{\widetilde{\pi}}(0)=\{\widetilde{z_1}, \ldots, \widetilde{z_r}, \widetilde{z}_{r+1}, \ldots, \widetilde{z_s}\}$. Then, due to the commutativity of 
	\begin{equation*}%\label{eine imp - diag reduktion} 
		\begin{tikzcd}
			Z \arrow[d, "p"] \arrow[rr, "\pi"] &  & \D^d \arrow[d, "p"] \\
			\wt Z \arrow[rr, "\wt \pi"]        &  & \mb A^d_k     
		\end{tikzcd}
	\end{equation*}
	and the fact that $\Dnull^d$ is the fiber over $0 \in \mb A^d_k$, we see that $\pi^{-1}(\Dnull^d)$ as an analytic space is the disjoint union of the $p^{-1}(\widetilde{z_i}), i=1,\ldots,s$. 
	Finally, each $p^{-1}(\widetilde{z_i})$ is connected by our generalized connectedness result Corollary \ref{cor faser zshgd}. Thus the connected components of $\pi^{-1}(\Dnull^d)$ are precisely the  $\invers{p}(\widetilde{z_1}), \ldots, \invers{p}(\widetilde{z_s})$. Since moreover $\mathring{W}=p^{-1}(\{\widetilde{z_1}, \ldots, \widetilde{z_r}\})$ by assumption, it follows that $\oo W$ is the union of the components $\invers{p}(\widetilde{z_1}), \ldots, \invers{p}(\widetilde{z_r})$.
\end{proof}

\begin{rmk}\label{rmk separabel}
	The morphism $\pi $ in Lemma \ref{eine implikation fuer spec wide op} can be chosen to be separable. Indeed,
		given a smooth connected affinoid space $Z=\Sp(R)$ and points $z_1, \ldots, z_n \in Z$, one can show (by a slight modification of the proof of \cite[Satz 4.1.9]{BKKN67}) that there exists a finite surjective separable morphism $\pi \colon Z \nach \D^d$ such that $\btrg{\pi(z_i)}<1$ for all $i=1, \ldots, n$. The property $\btrg{\pi(z_i)}<1$ means that $\wt{\pi}(\wt{z_i})=0$ (due to \cite[7.1.5/Proposition 2]{BGR84}), whence it follows by Lemma \ref{eine implikation fuer spec wide op} that $\pi$ has the other desired properties too.
\end{rmk}
\begin{rmk}
We mention that the converse of Lemma \ref{eine implikation fuer spec wide op} also holds: If  there exists a finite surjective morphism $\pi \colon Z \nach \mb D^d$ such that $\oo{W}$ is a  union of connected components of $\pi^{-1}(\mathring{\mb D}^d)$, then these connected components are fibers of the reduction map, so in particular $\oo W$ is special affinoid wide-open.
\end{rmk}
\begin{lemma}\label{lem1}
	Let $X$ be a connected smooth Stein space of dimension $d$. Then $X$ has a cover $\{\mathring{W}_i\}_{i \in \N}$ consisting of admissible open, special affinoid wide-open subsets $\mathring W_i$ such that 
	\begin{enumerate}
		\item the ambient affinoid space $W_i \supseteq \mathring{W_i}$ in the sense of Definition \ref{def spec wide op} can be chosen to be connected and contained in $X$, 
		\item $\mathring{W}_i \subseteq \mathring{W}_{i+1}$, and
		\item $\mathring{W}_i$ is smooth of dimension $d$.
	\end{enumerate} 
\end{lemma}
\begin{proof}
	By definition, $X$ has an admissible open cover $\{{W}_i\}_{i \in \N}$ by connected affinoid subsets $W_i$ satisfying $W_i \Subset  W_{i+1}$. One checks easily that the condition $W_i \Subset  W_{i+1}$ ensures that the image of $W_i$ under the reduction map $p_{i+1} \colon W_{i+1} \nach \wt{W_{i+1}}$ is a finite subset of $\wt {W_{i+1}}$. Then we can define $\oo W_{i+1}$ as the preimage of this finite set under $p_{i+1}$ and it is immediate that this $\oo W_{i+1}$ satisfies the desired conditions. 
\end{proof}
\subsection{Beyer's trace map}
First we recall the definition of the map
$
	\res \colon H^d_c(\oo{\D}^d, \omega_{\oo{\D}^d}) \nach K
$
from \cite[Definition 2.1.1]{Bey97a}. We let $K \langle X_1^{-1}, \ldots, X_d^{-1}\rangle^\dagger$ denote the ring of overconvergent series of the form 
$
\sum_{i_1,\ldots,i_d \leq 0}r_{i_1,\ldots,i_d}X_1^{i_1} \cdots X_d^{i_d}, \,\, r_{i_1,\ldots,i_d} \in K
$ 
and switch to multi-index notation, in particular
$
\frac{dX_1}{X_1} \wedge \ldots \wedge \frac{dX_d}{X_d}=\frac{dX_1\wedge \ldots \wedge dX_d}{X_1 \cdots X_d}=:\frac{dX}{X}.
$
Then, choosing coordinates $X=(X_1, \ldots, X_d)$ on $\D^d$ yields an isomorphism $H^d_c(\oo{\D}^d, \omega_{\oo{\D}^d}) \cong K \langle X^{-1}\rangle^\dagger \cdot \frac{dX}{X} $  (cf. \cite[Corollary 1.2.5]{Bey97a}), which allows one to define $\res$ by the following formula
	\begin{equation} \label{res vorschrift} % equation + aligned is used so that the several equations are only numbered with one centered number
	\begin{aligned}
		\res \colon H^d_c(\oo{\D}^d, \omega_{\oo{\D}^d}) &\nach K \\
	\sum_{i\leq 0}r_{i}X^{i} \cdot  \frac{dX}{X}  & \auf r_0.
	\end{aligned}
\end{equation}
This 
is independent of the choice of coordinates $(X_1, \ldots, X_d)$ on $\D^d$ by \cite[Proposition 2.1.3]{Bey97a}.
We note that Beyer works with the basis $dX$ and then $\res$ projects onto the $(-1)$-th coefficient, but this coincides with the $0$-th coefficient with respect to the basis $dX/X$.
Another important ingredient in the construction of the trace map is the following map from \cite[p.\ 234]{Bey97a}:
\begin{definition}[The map $\sigma$]\label{def trace sigma}
	Let $Z=\Sp(R)$ be a {connected} 
	smooth affinoid space of dimension $d$. Let $\mathring W \subseteq Z$ be a special affinoid wide-open space with an associated finite surjective separable morphism 
$
		\pi \colon Z \nach \D^d
$
	according to Remark \ref{rmk separabel}.
	We denote the trace map of the corresponding finite field extension
	\begin{equation*}%\label{field extension for sigma}
		E=Q(T_d) \nachinj L=Q(R)
	\end{equation*}
	by $\Tr_{L/E}$.
	It induces a map on the $d$-forms
	\begin{equation} \label{sigma vorschrift} % equation + aligned is used so that the several equations are only numbered with one centered number
		\begin{aligned}
			\Omega^d_{L/K} = \Omega^d_{E/K} \otimes_E L & \stackrel{{\sigma}}\nach \Omega^d_{E/K} \\
			\omega \otimes b & \auf \Tr_{L/E}(b) \cdot \omega,
		\end{aligned}
	\end{equation}
	where $\Omega^d_{L/K} = \Omega^d_{E/K} \otimes_E L$ holds because $L/E$ is separable. Moreover,  $\Omega^d_{L/K}=Q(R) \tens{R} \Omega^d_{R/K}$ and $\Tr_{L/E}(R) \subseteq T_d$
	since $\pi$ is finite. Therefore,  $\sigma (\Omega^d_{R/K} )\subseteq \Omega^d_{T_d/K}$, i.e.\ $\sigma$ restrics to a map
$$
		\sigma \colon \Omega^d_{R/K} \nach \Omega^d_{T_d/K},
$$
	which corresponds to a map 
$$
		\sigma \colon \pi_* \omega_Z \nach \omega_{\D^d}.
$$
	We write $\sigma=\sigma_\pi$ when we want to stress the dependence on $\pi$. 
\end{definition}
Now we reproduce \cite[Definition 4.2.4]{Bey97a}:
\begin{definition}[Trace map for special affinoid wide-open spaces]\label{defi beyers trace map}
	Let $Z$ be a connected smooth affinoid space of dimension $d$. Let $\mathring W \subseteq Z$ be a special affinoid wide-open space with an associated finite surjective separable morphism 
$
		\pi \colon Z \nach \D^d
$
according to Remark \ref{rmk separabel}.
	Let $\oo Z:= \invers{\pi}(\oo{\D}^d)$. \emph{The trace map}
$$
		t=t_\pi \colon H^d_c(\mathring{W}, \omega_Z) \nach K
$$
	is defined as the following composite map (which is in fact independent of the choice of the finite surjective separable morphism $\pi$ by \cite[Corollary 4.2.11]{Bey97a}):
$$
		H^d_c(\mathring{W}, \omega_Z) \nachinj H^d_c(\mathring{Z}, \omega_Z) \isoauf H^d_c(\oo{\D}^d, \pi_*\omega_Z) \xrightarrow{\,H^d_c(\sigma)\,} H^d_c(\oo{\D}^d, \omega_{\oo{\D}^d}) \xrightarrow{\,\, \res \,\,} K
$$
	where the  the third map is induced by the $\sigma$ from Definition \ref{def trace sigma}. 
\end{definition}
The definition of the trace map is then extended to smooth Stein spaces, since they admit admissible covers consisting of special affinoid wide-open spaces:

\begin{definition}[Trace map for Stein spaces]\label{bey trace for stein def}
	Let $X$ be a connected smooth Stein space of dimension $d$. In the notation of Lemma \ref{lem1}, we have the trace maps 
$
		t_i \colon H^d_c(\mathring{W}_i, \omega_X) \nach K.
$
	Since the diagrams 
	\begin{equation*}\label{diag1}
		\begin{tikzcd}
			H^d_c(\mathring{W}_i, \omega_X) \arrow[rd, "t_i"'] \arrow[r] & H^d_c(\mathring{W}_{i+1}, \omega_X) \arrow[d, "t_{i+1}"] \\
			& K                     
		\end{tikzcd}
	\end{equation*}
	commute (by \cite[Corollary 4.2.12]{Bey97a}), the $t_i$ induce a map 
$$
		t \colon \varinjlim_i H^d_c(\mathring{W}_i, \omega_X)=H^d_c(X, \omega_X) \nach K.
$$
	Beyer shows in \cite[Proof of Satz 7.1]{Bey97b} that this $t$ satisfies Theorem \ref{thm serre duality} below. In particular, it follows by standard universal abstract nonsense %(cf.\ \cite[Proposition III.7.2]{hartshorne})
	that, up to a unique automorphism of $\omega_X$, $t$ is independent of the choice of the covering $\{\mathring{W}_i\}_{i \in \N}$. We write $t=t_X$ when we wish to emphasize the base space $X$.
	
	\begin{theorem}[Serre duality for smooth rigid Stein spaces]\label{thm serre duality}
		Let $X$ be a smooth rigid $K$-space of dimension $d$.  If $X$ is Stein,
		then there is a canonical trace morphism
$$
			t \colon H^d_c(X, \omega_X) \nach K
$$
		which has the following property: If $\mc F$ is a coherent sheaf on $X$, then the composite of the trace map $t$ with the canonical pairing
$$
			H^{d-i}_c(X, \mc F) \times \Ext^{i}_X(\mc F, \omega_X) \nach H^d_c(X, \omega_X)
$$
		induces an isomorphism of topological $K$-vector spaces
$$
			H^{d-i}_c(X, \mc F)^\vee \isoauf \Ext^{i}_X(\mc F, \omega_X)
$$
		for all $i \geq 0$. 
	\end{theorem}
	Here $H^{d-i}_c(X, \mc F)^\vee$ denotes the space of continuous linear forms on $H^{d-i}_c(X, \mc F)$, equipped with the strong dual topology. Moreover, $\Ext_{X}^{i}(\mc F, \omega_{X})$ is equipped with the canonical topology for global sections of a coherent sheaf (see \cite[\S 1.3]{Bey97a}), as $\Ext_{X}^{i}(\mc F, \omega_{X})=H^0(X, \sExt_{X}^{i}(\mc F, \omega_{X}))$ because the spectral sequence for the derived functor of the
	composition is degenerate. Indeed, the spectral sequence degenerates since $X$ is quasi-Stein and $\sExt_{X}^{i}(\mc F, \omega_{X})$ is  a coherent $\mc O_{X}$-module for all $i$ (cf.\ \cite[Proposition 3.3 and also the discussion preceding Lemma 3.7]{Chi90}).
\end{definition}

\subsection{The residue map on local cohomology}\label{subsec: res map}

Let $Z=\Sp(R)$ be a connected smooth affinoid space of dimension $d$ and $z \in Z$ a point with corresponding maximal ideal $\mf{m}_z \subseteq R$. 
Following Beyer, we often shorten notation as follows: Given a coherent sheaf $\mc F$ on $Z$, we set $M=\Gamma(Z, \mc F)$ and
$$
	H^j_z(\mc F):= H^j_{\mf{m}_z}(M)=H^j_{\wh{{\mf m}_z}}(\wh M),
$$
%where $M=\Gamma(Z, \mc F)$.
where the latter identification between the local cohomology groups (\quot{insensitivity to completion}) is shown\footnote{The assertion \cite[Proposition 2.15]{Hun07} concerns the case of a local ring and its unique maximal ideal, but the proof carries over verbatim to our (more general) setting of the Noetherian ring $R$ and any maximal ideal ${\mf{m}}_z$.} in \cite[Proposition 2.15]{Hun07}. 

There are two fundamental properties of local cohomology that play an important role throughout. One is its relation to the sheaf cohomology of the underlying affine scheme $\Spec(R)$ as treated in \cite[Theorem 12.47]{ILL+07}, which ultimately yields the link to compactly supported cohomology described in $\S$\ref{subsec relation} below. The other is the explicit description of local cohomology as a colimit of Koszul complexes:
\begin{lemma}\label{koszul coho}
		Let $A$ be a Noetherian ring of dimension $d$, $\mf a \subseteq A$ an ideal and $M$ an $A$-module. Let $t_1, \ldots, t_d \in A$ be such that $\sqrt{\mf a}=\sqrt{(t_1, \ldots, t_d)}$. Write $t=(t_1, \ldots, t_d)$. For each $\rho \in \N$, set $t^\rho:=(t_1^\rho, \ldots, t_d^\rho)$. Then there is a canonical isomorphism
$$
		\varinjlim_\rho M/t^\rho M \isoauf H^d_{\mf a}(M),
$$
where the transition map $M/t^{\rho_1} M \nach M/t^{\rho_2} M$ for $\rho_1\leq \rho_2$ on the left-hand side is given by $m \auf (\prod_{i=1}^{d}t_i)^{\rho_2-\rho_1} \cdot m$.
\end{lemma}
\begin{proof}
Note that, when $A$ is local and $\mf a= \mf m$ its unique maximal ideal, the condition $\sqrt{(t_1, \ldots, t_d)}=\sqrt{\mf m}=\mf m$ means that $(t_1, \ldots, t_d)$ is a system of parameters (by definition \doublecitestacks{07DU}{00KU}). This special case is the content of \cite[Lemma 3.1.2]{Bey97a}. For the more general statement, we may assume that $\mf a =(t_1, \ldots, t_d)$ due to the insensitivity of local cohomology to taking radicals \cite[Proposition 7.3 (2)]{ILL+07}, and then conclude from \cite[Theorem 7.11]{ILL+07}.
\end{proof}
For instance, this result yields the isomorphisms  (\ref{beyersisoauflocalcoho1}) and (\ref{beyersisoauflocalcoho2}) further below. \\
 Now, in \cite[Definition 4.2.7]{Bey97a}, Beyer defines a canonical \textit{residue map} 
\begin{equation}\label{res auf local coho}
	\res_z \colon H^d_z(\omega_Z)= H^d_{\mf{m}_z}(\Omega^d_{R/K}) \nach K
\end{equation}
whose explicit construction we need not recall for our purposes, but rather its properties:

\begin{lemma}\label{lemma gamma new}
	Let $Z=\Sp(R)$ be a {connected} smooth affinoid space of dimension $d$. Let $\mathring W \subseteq Z$ be a special affinoid wide-open space with an associated finite surjective morphism 
$
		\pi \colon Z \nach \D^d
$
	according to Lemma \ref{eine implikation fuer spec wide op}, and let $\varphi \colon T_d \nach R$ be the finite injective ring morphism corresponding to $\pi$. Let $\{z_1, \ldots, z_r\}=\pi^{-1}(0) \cap \mathring W $ and $\{z_1, \ldots, z_r, z_{r+1}, \ldots z_s\}=\pi^{-1}(0)$. Denote by $\mf m_1, \ldots \mf m_s \subseteq R$ the corresponding maximal ideals in $R$. Let $M=\Gamma(Z, \omega_Z)$ and let $\mf m \subseteq T_d$ denote the maximal ideal corresponding to $0 \in \D^d$.
	Then:
	\begin{enumerate}[(i)]
		\item For every coherent sheaf $\mc F$ on $Z$, there is a canonical isomorphism
$$
			\gamma \colon \bigoplus\limits_{i=1}^s H_{z_i}^d(\mc F) \isoauf H_0^d(\pi_* \mc F). 
$$
		We write $\gamma=\gamma_{\mc F, \pi}$ when we want to stress the dependence on $\mc F$ and $\pi$. \label{rm lc may vie}
		\item {Let $X_1, \ldots , X_d$ be a system of parameters for the $\mf m$-adic completion ${T_d}^{\wedge \mf m}$}. There are canonical isomorphisms 
		\begin{equation}\label{beyersisoauflocalcoho1} 
			H_0^d(\pi_* \omega_Z) \cong \varinjlim_\rho \wh{M_{\mf m}} / (X_1^\rho, \ldots, X_d^\rho)
		\end{equation}
		and
		\begin{equation}\label{beyersisoauflocalcoho2}
			\bigoplus\limits_{i=1}^s H_{z_i}^d(\omega_Z) \cong \varinjlim_\rho \bigoplus\limits_{i=1}^s \wh{M_{\mf m_i}} / (X_1^\rho, \ldots, X_d^\rho),
		\end{equation}
		where we have denoted the image of $X_j$ under each map ${T_d}^{\wedge \mf m} \nach R^{\wedge \mf m_i}$ induced by $\varphi$ on completions again by $X_j$; these form a system of parameters in $R^{\wedge \mf m_i}$.
		Via these isomorphisms, $\gamma^{-1}$ is identified with the map
		\begin{equation}\label{gammatilde vorschrift}
		\begin{aligned}
				\widetilde{\gamma}^{-1} \colon \varinjlim_\rho \wh{M_{\mf m}} / (X_1^\rho, \ldots, X_d^\rho) & \isoauf \varinjlim_\rho \bigoplus\limits_{i=1}^s \wh{M_{\mf m_i}} / (X_1^\rho, \ldots, X_d^\rho) \\
			\begin{bmatrix}
				\omega \\
				X^\rho
			\end{bmatrix} & \auf \left (\begin{bmatrix}
				\omega_1 \\
				X^\rho
			\end{bmatrix}, \ldots, \begin{bmatrix}
				\omega_r \\
				X^\rho
			\end{bmatrix} \right )
		\end{aligned}
		\end{equation}
		where $\left [ \begin{smallmatrix}
			\omega \\
			X^\rho
		\end{smallmatrix} \right ]$ denotes the image of $\omega$ under $\wh{M_{\mf m}} / (X_1^\rho, \ldots, X_d^\rho) \nach \varinjlim_j \wh{M_{\mf m}} / (X_1^j, \ldots, X_d^j),$ and $\omega_i$ denotes the image of $\omega$ in $\wh{M_{\mf m_i}}$. \label{rm lc darst} %\tr{Wie sieht $\tilde{\gamma}$ aus? Einfach die Komponenten aufsummieren?} \tp{Evtl.\ irrelevant tho...}
		\item \label{rm lc dreieck nummer}
		The diagram 
		\begin{equation*}% \label{rm lc dreieck}
			\begin{gathered}
				\xymatrix{
					&\bigoplus\limits_{i=1}^s H_{z_i}^d(\omega_Z) \ar[r]^-{\gamma}_{\sim} \ar[rd]_-{\sum \res_{z_i}} & H_0^d(\pi_*\omega_{Z}) \ar[r]^-{H_0^d(\sigma)} & H_0^d(\omega_{\D^n}) \ar[ld]^-{\res_0} \\
					& & K & }
			\end{gathered}
		\end{equation*}
		commutes, where $\sigma=\sigma_\pi$ is defined as in Definition \ref{def trace sigma} above.  
		%\ref{res seprbl} 
	\end{enumerate}
\end{lemma}
\begin{proof}
	\begin{enumerate}[(i)]
		\item The isomorphism $\gamma$ is the map from \cite[Lemma 4.2.9 (a)]{Bey97a}.
		\item This is true by the arguments in \cite[Proof of Lemma 4.2.9 (c)]{Bey97a}. %\tp{For the sake of completeness, we explain how one obtains $\gamma$, since this is not done in \cite{beyer1997serre}.}
		\item This assertion is precisely \cite[Lemma 4.2.9 (c)]{Bey97a}. 
	\end{enumerate}
\end{proof}

\subsection{The map from local cohomology into compactly supported cohomology}\label{subsec relation}
Let $Z=\Sp(R)$ be a {connected} smooth affinoid space of dimension $d$, let $\mathring W \subseteq Z$ be a special affinoid wide-open space. Given a point $z \in \oo W$ and a coherent sheaf $\mc F$ on $Z$, \cite[Lemma 4.2.6]{Bey97a} constructs a canonical map
$$
H_{z}^d(\mc F) \nach H^d_c(\oo{W}, \mc F) 
$$
	that is functorial in $\mc F$. Again, we need not recall its explicit construction for our purposes, but rather the following properties:
\begin{lemma}\label{lemma interplay}
	Let $Z=\Sp(R)$ be a {connected} smooth affinoid space of dimension $d$. Let $\mathring W \subseteq Z$ be a special affinoid wide-open space. 
%	Then there exist points $z_1,\ldots,z_s$ in $\oo W$ such that the image of the map
%	$$
%	\bigoplus\limits_{i=1}^s H_{z_i}^d(\omega_Z) \nach H^d_c(\oo W, \omega_Z) 
%	$$
%	is dense in $H^d_c(\oo W, \omega_Z)$. \\
%	More precisely: 
	Let
$
		\pi \colon Z \nach \D^d
$
	be any finite surjective morphism associated to $\oo W$ as in Lemma \ref{eine implikation fuer spec wide op}.
	Let $\{z_1, \ldots, z_r\}=\pi^{-1}(0) \cap \mathring W $ and $\{z_1, \ldots, z_r, z_{r+1}, \ldots z_s\}=\pi^{-1}(0)$. Set
	$
	\oo Z = \invers{\pi}(\Dnull^d).
	$
	Then
	the image of the canonical map
	$$
	\bigoplus\limits_{i=1}^s H_{z_i}^d(\omega_Z) \nach H^d_c(\oo{Z}, \omega_Z) 
	$$
	is dense in $H^d_c(\oo{Z}, \omega_Z)$ and the image of the following map induced by restriction
	$$
	\bigoplus\limits_{i=1}^r H_{z_i}^d(\omega_Z) \nach H^d_c(\oo W, \omega_Z) 
	$$
	is dense in $H^d_c(\oo W, \omega_Z)$. 
\end{lemma}
\begin{proof}
	This is \cite[Lemma 4.2.9 (d)]{Bey97a}. 
\end{proof}

\section{The theorem on the relative trace map} \label{sec: rel trace}
In this section, we introduce the relative trace map, then state and prove the main theorem about it. 
Throughout this section, let $\alpha \colon X \nach Y$ be a finite étale morphism of smooth $d$-dimensional rigid spaces. 
\subsection{The relative trace map $t_\alpha$}
We start by defining the relative trace map.	Because $\alpha$ is finite, it pulls back affinoids to affinoids. Consequently, a result analogous to \citestacks{01SB} (see also \ {\cite[Proposition I.9.2.5]{Gro60}}) holds, namely: any coherent $\alpha_* \mc O_X$-module $\mc M$ can naturally be regarded as a coherent $\mc O_X$-module $\widetilde{\mc M}$ such that $\alpha_* \widetilde{\mc M}=\mc M$
and $(-)^{\sim} $ is an equivalence of categories\footnote{The $\mc O_X$-module sheaf $\wt{\mc M}$ is constructed by gluing as follows: If $V \subseteq Y$ is affinoid open, then $\mc M(V)$ is a module over $\alpha_*\mc{O}_X(V)=\mc O_X(\invers{f}(V))$. Thus, on the affinoid open subspace $\invers{f}(V) \subseteq X$, we can consider the  $\mc O_X$-module sheaf associated to $\mc M(V)$. Call this sheaf $\wt{\mc{M}}|_{\invers{f}(V)}$. Then one checks that the $\wt{\mc{M}}|_{\invers{f}(V)}$ glue to an $\wt{\mc{M}}$, when $V$ runs through an affinoid open cover of $Y$.}.
Since $\alpha$ is étale, there is a natural isomorphism
\begin{equation}\label{def rel trace erster iso}
	(\alpha_* \mc O_X \otimes_{\mc O _Y} \omega_Y)^{\sim} \isoauf \omega_X.
\end{equation}
%	where $\omega_Y$ denotes the sheaf of holomorphic $d$-forms on $Y$. 
Since $\alpha$ is finite flat, there is the usual trace pairing 
%(cf.\ Section \ref{subsec another interpretation})
\begin{equation}\label{def rel trace dritter iso}
	\alpha_* \mc O_X \nach {\sHom}_Y(\alpha_* \mc O_X, \mc O_Y).
\end{equation}
Finally, since $\alpha$ is finite flat, the natural map
\begin{equation}\label{def rel trace zweiter iso}
	{\sHom}_Y(\alpha_* \mc O_X, \mc O_Y) \otimes_{\mc O_Y} \omega_Y \isoauf {\sHom}_Y(\alpha_* \mc O_X, \omega_Y)
\end{equation}
is an isomorphism. 
{\emph{The relative trace map}} is now defined to be the composite map
\begin{align*}
	\alpha_* \omega_X & \stackrel{(\ref{def rel trace erster iso})}\cong \alpha_*(\alpha_* \mc O_X \otimes_{\mc O _Y} \omega_Y)^{\sim} \\ & \stackrel{(\ref{def rel trace dritter iso})}\nach \alpha_*({\sHom}_Y(\alpha_* \mc O_X, \mc O_Y) \otimes_{\mc O _Y} \omega_Y)^{\sim} \\ & \stackrel{(\ref{def rel trace zweiter iso})} \cong
	\alpha_*{\sHom}_Y(\alpha_* \mc O_X, \omega_Y)^{\sim} \\ & \stackrel{\phantom{(\ref{def rel trace zweiter iso})}} =
	{\sHom}_Y(\alpha_* \mc O_X, \omega_Y) \xrightarrow{g \auf g(1)} \omega_Y
\end{align*}
and is denoted by $t_\alpha$.

We aim for a more down-to-earth description of the relative trace map in terms of affinoids. Let $V=\Sp(A)$ be an affinoid in $Y$ and $U=\invers{\alpha}(V)=\Sp(B)$ its preimage in $X$. The associated ring morphism $A \nach B$
is, in particular, finite and flat. 
Since $A$ is Noetherian, the finite flat ring morphism $A \nach B$ makes $B$ into a finitely presented flat $A$-module, i.e.\ a finitely generated projective $A$-module. Hence the natural map 
\begin{align*}
	\tn{can} \colon B^\ast \tens A B & \isoauf \End_A(B) \\
	f \otimes b & \auf [x \auf f(x) \cdot b]
\end{align*}
is an isomorphism, where $B^*:=\Hom_A(B,A)$. The trace $\Tr_{B/A}$ is now defined as the composite
\[
\begin{tikzcd}
	B \arrow[r] \arrow[rrr, "\Tr_{B/A}"', bend right] & \End_A(B)  \arrow[r, "\invers{\tn{can}}"] &  B^* \otimes_A B \arrow[r] & A
\end{tikzcd}
%\begin{tikzcd}
%B \arrow[r] \arrow[rr, "\Tr_{B/A}"', bend right] & \End_A(B) \cong B^* \otimes_A B \arrow[r] & A
%\end{tikzcd}
\]
where the first map sends $b \in B$ to the endomorphism given by multiplication by $b$, and the last map is given by ``evaluation''. If $B$ is free of finite rank over $A$, then $\Tr_{B/A}$ coincides with the usual trace map from linear algebra.  Note that the map (\ref{def rel trace dritter iso}) of coherent sheaves boils down to $B \nach B^*, b \auf \Tr_{B/A}(b \cdot (-))$ at the level of modules.  Therefore, the composite
	\[
\begin{tikzcd}
	\Omega_{B/K}^d \arrow[r, phantom] \arrow[r, "\sim", "(\ref{def rel trace erster iso})"'] & B \otimes_A \Omega_{A/K}^d \arrow[r, "(\ref{def rel trace dritter iso})"'] & B^* \otimes_A \Omega_{A/K}^d \arrow[r, "\sim", "(\ref{def rel trace zweiter iso})"'] & {\Hom_A(B, \Omega_{A/K}^d)} \arrow[r, "f \auf f(1)"] & \Omega_{A/K}^d
\end{tikzcd}
\]
%(which, by definition, corresponds to $t_\alpha \colon \alpha_* \omega_U \nach \omega_V$)
 coincides with the following map 
	\begin{equation}\label{tau vorschrift}
		\begin{aligned}
			\Omega^d_{B/K} = \Omega^d_{A/K} \otimes_A B & \stackrel{\tau}\nach \Omega^d_{A/K} \\
			\omega \otimes b & \auf \Tr_{B/A}(b) \cdot \omega.
		\end{aligned}
	\end{equation}
We record this fact in the following lemma, for ease of reference.
\begin{lemma}\label{tau ist rel trace}
	The restriction of the relative trace map $t_\alpha \colon \alpha_* \omega_U \nach \omega_V$ to affinoids $U=\Sp(B)$ and $V=\Sp(A)$ is associated to the homomorphism of modules $\tau \colon \Omega^d_{B/K} \nach \Omega^d_{A/K}$ from (\ref{tau vorschrift}).
\end{lemma}
\begin{lemma}\label{lem quotkp}
	If $R'$ is an overring of an integral domain $R$ such that each $r'\in R'$ is integral over $R$ and such that no element of $R \setminus \{0\}$ is a zero divisor in $R'$, then the localization  $R'_{R\setminus \{0\}}$ of $R'$ at the multiplicative subset $R \setminus \{0\}$ coincides with the total ring of fractions $Q(R')$.
\end{lemma}
\begin{proof}
	This is proved in the discussion after \cite[3.1.3/Proposition 3]{BGR84}.
\end{proof}
Suppose that there are finite surjective separable morphisms
$
\pi_U \colon U \nach \D^d
$ and $\pi_V \colon V \nach \D^d$ such that $\pi_U=\pi_V \circ \alpha$. Then there is the following compatibility of $\tau$  with $\sigma_U:=\sigma_{\pi_U}$ and $\sigma_V := \sigma_{\pi_V}$ from Definition \ref{def trace sigma}: 
\begin{lemma}\label{lemma sigma}
	The diagram 
	\begin{align*}
		\xymatrix{
			&\Omega^d_{B/K}  \ar[r]^-{\sigma_U} \ar[d]_-{\tau} & \Omega^d_{T_d/K} \\
			&\Omega^d_{A/K}  \ar[ru]_-{\sigma_V}  } 
	\end{align*}
	commutes. By Lemma \ref{tau ist rel trace}, it corresponds to the commutative diagram 
	\begin{align*}
		\xymatrix{
			&{\pi_U}_* \omega_U  \ar[r]^-{\sigma_U} \ar[d]_-{{\pi_V}_*(t_\alpha)} & \omega_{\D^d}\\
			&{\pi_V}_* \omega_V   \ar[ru]_-{\sigma_V}.  } 
	\end{align*}
\end{lemma}
\begin{proof}
	In the following, we write  $Q(B)^*=\Hom_{Q(A)}(Q(B),Q(A))$. Due to the formulae (\ref{sigma vorschrift}) and (\ref{tau vorschrift}), we see that it suffices to show that 
	\begin{equation}\label{trace eqlty}
		\Tr_{Q(B)/Q(T_d)}|_B={\Tr_{Q(A)/Q(T_d)}}|_A \circ \Tr_{B/A}.
	\end{equation}
	We will show that the restriction of $\Tr_{Q(B)/Q(A)}$ to $B$ coincides with $\Tr_{B/A}$, i.e.\ that the diagram
	\[
	\begin{tikzcd}
		Q(B) \arrow[r]        & \End_{Q(A)}(Q(B)) \arrow[r, "\sim"] & Q(B)^* \otimes_{Q(A)} Q(B) \arrow[r] & Q(A)        \\
		B \arrow[u] \arrow[r] & \End_A(B) \arrow[r, "\sim"]         & B^* \otimes_A B \arrow[r]            & A \arrow[u]
	\end{tikzcd}
	\]
	commutes, whence the desired equality (\ref{trace eqlty}) follows by the transitivity of the trace in towers of field extensions. Lemma \ref{lem quotkp} tells us that $Q(B)=B \tens{A} Q(A)$, so extension of scalars yields the maps
	\begin{alignat*}{3}
		\End_A(B) & \nach \End_{Q(A)}(Q(B)) \quad \quad \tn{  and   } \quad \quad & B^* & \nach Q(B)^*  \\ 
		f & \auf f_Q:=f \otimes \id_{Q(A)}  &	f & \auf f_Q:=f \otimes \id_{Q(A)}.
	\end{alignat*}
	Thus we can expand the above diagram to
	\[
	\begin{tikzcd}
		Q(B) \arrow[r]        & \End_{Q(A)}(Q(B)) \arrow[r, "\sim"]   & Q(B)^* \otimes_{Q(A)} Q(B) \arrow[r] & Q(A)        \\
		B \arrow[u] \arrow[r] & \End_A(B) \arrow[r, "\sim"] \arrow[u] & B^* \otimes_A B \arrow[r] \arrow[u]  & A. \arrow[u]
	\end{tikzcd}
	\]
	Now it is easy to see that each of the three squares in the diagram commute. Indeed, to see, for instance, that the middle square commutes, note that the isomorphism $Q(B) \cong B \tens{A} Q(A)$ means that any element of $Q(B)$ can be written as $\frac{x}{a}$ with $x \in B$ and $a \in A$, and any $\phi_Q$ then acts on it as $\phi_Q(\frac{x}{a})=\frac{\phi(x)}{a}$. The commutativity of the middle square amounts to showing that, if $\phi(x)=\sum_i f_i(x)b_i$ for all $x \in B$, then $\phi_Q(\frac{x}{a})=\sum_i{f_i}_{Q}(\frac{x}{a})\frac{b_i}{1}$ for all $x \in B$ and all $a \in A$. But $\phi_Q(\frac{x}{a})=\frac{\phi(x)}{a}$ and $\sum_i{f_i}_{Q}(\frac{x}{a})\frac{b_i}{1}=\sum_i \frac{f_i(x)}{a}\frac{b_i}{1}=\frac{\sum_i f_i(x)b_i}{a}=\frac{\phi(x)}{a}$ as well, as desired.
\end{proof}
The following {three subsections} are devoted to proving our main result in this section:
\begin{theorem}\label{simpl thm rel tr}
Denote the composite map $H^d_c(X, \omega_X) \isoauf   H^d_c(Y, \alpha_*\omega_X)\xrightarrow{H^d_c(Y,t_\alpha)}   H^d_c(Y, \omega_Y)$ by $q_\alpha$. 
	Then the following diagram commutes:
	\begin{equation}\label{diag0}
		\begin{gathered}
			\xymatrix{
				&H^d_c(Y, \omega_Y)  \ar[rd]^-{t_Y} \\
				&    &K. \\
				&H^d_c(X, \omega_X)  \ar[ru]_-{t_X} \ar[uu]^-{q_\alpha} } 
		\end{gathered}
	\end{equation} 
%	commutes.
\end{theorem}
\subsection{Reducing the proof of Theorem \ref{simpl thm rel tr} to special affinoid wide-opens}\label{sec comp coverings}
To reduce the proof of Theorem \ref{simpl thm rel tr} to the case of special affinoid wide-opens, we need the following ``relative version'' of Lemma \ref{lem1}. Recall that $\alpha \colon X \nach Y$ is a finite étale morphism of smooth $d$-dimensional rigid spaces. 
\begin{lemma}\label{lemma: banger}
	There exist admissible covers $\{\mathring{U}_i\}_{i \in \N}$ of $X$ and $\{\mathring{V}_i\}_{i \in \N}$ of $Y$ by special affinoid wide-opens as in Lemma \ref{lem1} with ambient affinoid spaces $U_i \subseteq X, V_i \subseteq Y$, such that $\alpha(\mathring{U}_i) \subset \mathring{V}_i$ and $\alpha(U_i) \subseteq V_i$ and $\alpha \colon {U}_i \nach {V}_i$ is finite étale. More precisely, there is a commutative diagram
	\[
	\begin{tikzcd}
		\mathring{V_i} \arrow[r, "\subseteq"]                     & V_i \arrow[r, "\pi_{V_i}", two heads]                                  & \D^d \\
		\mathring{U_i} \arrow[u, two heads] \arrow[r, "\subseteq"] & U_i \arrow[u, "\alpha", two heads] \arrow[ru, "\pi_{U_i}"', two heads] &  
	\end{tikzcd}
	\]
	for each $i$, where $U_i \subseteq X$ and $V_i \subseteq Y$ are connected smooth affinoids, $\alpha$ is finite étale, $\pi_{V_i}$ and $\pi_{U_i}$ are finite and separable and are associated to $\oo V_i$ resp.\ $\oo U_i$ (in the sense of Lemma \ref{eine implikation fuer spec wide op}), and all arrows with two heads are surjective. 
\end{lemma}
\begin{proof}
	We obtain the existence of $\{\mathring{V}_i\}_{i \in \N}$ by Lemma \ref{lem1}. Moreover, we have a connected %and that $\mathring{V}_i=\pi_{V_i}^{-1}(\Dnull^d)$.} 
und smooth affinoid space $V_i=\Sp(A_i)$  with $\mathring{V}_i \subseteq V_i \subseteq Y$, such that $\oo V_i$ is the preimage of finitely many points under the reduction map $p_{V_i}$, say $$\oo V_i=p^{-1}_{V_i}(\{\widetilde{v}_1, \ldots, \widetilde{v}_r\}).$$
The preimage of an affinoid space under a finite morphism is again affinoid, so $U_i':=\alpha^{-1}({V}_i)$ is affinoid. Moreover, the restriction $\alpha \colon U_i' \nach V_i$ is also finite étale. Replacing $U_i'$ with one
of its connected components $U_i$ (which is an affinoid subdomain in $U_i'$, say $U_i=\Sp(B_i)$), the restriction $\alpha \colon U_i \nach V_i$ is again finite étale. Here we use that $U_i$ is \quot{clopen} in $U_i'$ (being a connected component): the restriction is again étale since it arises by composition with the open immersion $U_i \nachinj U_i'$ (which is étale) and, on the other hand, it is again finite since it arises by composition with the closed immersion $U_i \nachinj U_i'$ (which is finite).  The restriction $\alpha \colon U_i \nach V_i$ being finite étale, the associated ring morphism $A_i \nach B_i$ is finite flat. %\tp{Reference?? I'm sure this is true tho, eine Referenz wird sich finden lassen :)}.
We may assume that $U_i \neq \emptyset$. Indeed, since $X \neq \emptyset$, we see that $U_i \neq \emptyset$ for all $i  \gg 0$, so we may re-index and forget the small $i$. Next we argue that $A_i \nach B_i$ is injective. This is true since any flat ring morphism $R \nach S$ from an integral domain $R$ to a ring $S \neq 0$ is necessarily injective, and $A_i$ is a domain since $V_i$ is connected. So $A_i \nach B_i$ is an injective integral morphism, which implies that the map $\Spec(B_i) \nach \Spec(A_i)$ is surjective. Thus $$\alpha \colon U_i \nach V_i$$ is surjective as well, because any prime ideal lying over a maximal ideal in an integral ring extension is necessarily maximal. {Note that $\alpha$ being finite implies that $\wt \alpha$ is also finite (by \cite[6.3.5/Theorem 1]{BGR84}) and that $\alpha$ being surjective implies that $\wt \alpha$ is also surjective (since both $\alpha$ and $p_{V_i}$ are surjective and thus also the composite $p_{V_i} \circ \alpha=\wt{\alpha} \circ p_{U_i}$ is surjective}). 
Since $\wt \alpha$ is finite, $\Omega:=\wt \alpha^{-1}(\{\widetilde{v}_1, \ldots, \widetilde{v}_r\})$ is a finite subset of $\wt U_i$ and hence $$\oo U_i:=p^{-1}_{U_i}(\Omega)$$ is special affinoid wide-open in $U_i$. Then $\alpha(\oo U_i) \subseteq \oo V_i$  because $p_{V_i} \circ \alpha = \wt \alpha \circ p_{U_i}$. In fact, we claim that $\oo U_i = \invers{\alpha}(\oo V_i)$, so the restriction $\oo U_i \nach \oo V_i$  of $\alpha$ is again surjective. To see this, let $u \in \invers{\alpha}(\oo V_i)$, which by definition of $\oo V_i$ means that 
$$
\{\widetilde{v}_1, \ldots, \widetilde{v}_r\} \ni p_{V_i}(\alpha(u))=\wt{\alpha}(p_{U_i}(u)),
$$ 
so $p_{U_i}(u) \in \wt \alpha^{-1}(\{\widetilde{v}_1, \ldots, \widetilde{v}_r\})=\Omega$, i.e.\ $u \in \invers{p_{U_i}}(\Omega)=\oo U_i$ which proves our claim. In particular, it follows that $\{\mathring{U}_i\}_{i \in \N}$  is the preimage of the admissible open cover $\{\mathring{V}_i\}_{i \in \N}$ of $Y$ under $\alpha \colon X \nach Y$, hence it is itself an admissible open cover of $X$. Now take a finite surjective map $\wt \pi_{V_i}\colon \wt{V_i} \nach \mb A^d$ with $\wt \pi_{V_i}(\wt{v}_j)=0$ for all $j=1, \ldots, r$ and a separable lift $\pi_{V_i} \colon V_i \nach \D^d$ (see Remark \ref{rmk separabel}). Define $$\wt \pi_{U_i}:=\wt \pi_{V_i} \circ \wt \alpha.$$
Then $\wt \pi_{U_i}(\Omega)=\wt \pi_{V_i}(\wt \alpha(\Omega)) \subseteq \wt \pi_{V_i}(\{\widetilde{v}_1, \ldots, \widetilde{v}_r\}) = \{0\}.$ {Moreover, $\wt \pi_{U_i}$ is finite surjective since $\wt \pi_{V_i}$ and $\wt \alpha$ are both finite and surjective}. Therefore, any separable lift of $\wt \pi_{U_i}$ satisfies the conditions in the assertion, by Lemma \ref{eine implikation fuer spec wide op}. In particular, the separable lift
$$\pi_{U_i}:= \pi_{V_i} \circ \alpha$$
satisfies the desired conditions. Here we used that $\alpha$ is separable, which is true because it is unramified at all closed points of $U$ and hence necessarily generically unramified.
\end{proof}
Lemma \ref{lemma: banger} reduces us to showing that the diagram 
\begin{equation}\label{dreieck1}
	\begin{gathered}
		\xymatrix{
			&H^d_c(\mathring{V}_i, \omega_Y)  \ar[rd]^-{t_{\mathring{V}_i}} \\
			&    &K \\
			&H^d_c(\mathring{U}_i, \omega_X)  \ar[ru]_-{t_{\mathring{U}_i}} \ar[uu]^-{q_\alpha} } 
	\end{gathered}
\end{equation}
commutes for all $i$. Indeed, taking $\varinjlim_i$ then yields Theorem \ref{simpl thm rel tr} (see Definition \ref{bey trace for stein def}).
\subsection{Utilizing local cohomology}
Consider the diagram from Lemma \ref{lemma: banger} for a fixed $i$, but omit the index $i$ in the notation. Thus $\alpha\colon \oo U \nach \oo{V}$ is the restriction of a surjective finite étale morphism 
$$
	\alpha \colon U=\Sp(B) \nach V=\Sp(A).
$$
The associated ring morphism $A \nach B$
is, in particular, finite and flat. We consider the commutative diagram
\begin{equation}\label{diag dreieck}
	\begin{gathered}
		\begin{tikzcd}
			V \arrow[r, "\pi_V", two heads]                                  & \D^d & \tn{associated to} & A \arrow[d, hook] & {T_d,} \arrow[l, hook] \arrow[ld, hook] \\
			U \arrow[u, "\alpha", two heads] \arrow[ru, "\pi_U"', two heads] &      &          & B                 &                                                                          
		\end{tikzcd}
	\end{gathered}
\end{equation}
where all morphisms are finite and $A$ and $B$ are integral domains. This notation is fixed for the remainder of this section. 
Let
$$
	\{x_1, \ldots, x_r\}=\pi_{U}^{-1}(0) \cap \mathring{U} \quad \tn{   and     } \quad \{x_1, \ldots, x_r,x_{r+1},\ldots, x_s\}=\pi_{U}^{-1}(0).
$$
Note that $\{\alpha(x_1), \ldots, \alpha(x_r)\} \subseteq \pi_{V}^{-1}(0) \cap \mathring{V}$, since $\pi_U=\pi_V \circ \alpha$ by Lemma \ref{lemma: banger}. Moreover, $\{\alpha(x_1),\ldots, \alpha(x_s)\}=\pi_{V}^{-1}(0)$ since $\pi_U=\pi_V \circ \alpha$ and $\alpha$ is surjective. Denote the cardinality of $\{\alpha(x_1),\ldots, \alpha(x_s)\}$ by $s'$ and let 
$$\{y_1, \ldots, y_{s'}\}=\{\alpha(x_1),\ldots, \alpha(x_s)\}=\pi_{V}^{-1}(0)$$ 
with $y_i \neq y_j$ for $i \neq j$. Then $s' \leq s$ and it may happen that $s'<s$. Define $r'$ in the same way for $\{\alpha(x_1), \ldots, \alpha(x_r)\}$ and arrange the ordering of $\{y_1, \ldots, y_{s'}\}$
in the right way so that $\{y_1, \ldots, y_{r'}\}$ equals
$\{\alpha(x_1),\ldots, \alpha(x_r)\}$.
We bring local cohomology into the game: Using the maps from $\S\S$ \ref{subsec: res map}-\ref{subsec relation}, expand the diagram (\ref{dreieck1}) %as in \cite[Lemma 4.2.9]{beyer1997serre} 
to 
\begin{equation}\label{erstes big diag}
	\begin{gathered}
		\xymatrix{
			& \bigoplus\limits_{i=1}^{r'} H_{y_i}^d(\omega_V) \ar@/^5pc/[rrd]^{\sum \tn{res}_{y_i}} \ar[r] &H^d_c(\mathring{V}, \omega_V)  \ar[rd]^-{t_{\mathring{V}}} \ar@{}[rdd]|{?} \\
			& &   &K \\
			&\bigoplus\limits_{i=1}^r H_{x_i}^d(\omega_U) \ar[r]^-{\textnormal{dense}} \ar@/_5pc/[rru]_{\sum \tn{res}_{x_i}} &H^d_c(\mathring{U}, \omega_U)  \ar[ru]_-{t_{\mathring{U}}} \ar[uu]^-{q_\alpha} & } 
	\end{gathered}
\end{equation}
where the lower horizontal map has dense image by Lemma \ref{lemma interplay}. %\tgray{Note that it may happen that $y_i = y_j$ for $i \neq j$, but this doesn't pose a problem.}  
We have placed a question mark in the triangle in (\ref{erstes big diag}) for psychological reasons - as a reminder that we need to show that the triangle commutes.
The two outer ``slices''
\begin{align*}
	\xymatrix{
		& & & & & & K \\
		& \! \! \! \! \! \! \! \! \! \! \! \! \! \!  \bigoplus\limits_{i=1}^{r'} H_{y_i}^d(\omega_V) \ar@/^5pc/[rrd]^{\sum \tn{res}_{y_i}} \ar[r] & H^d_c(\mathring{V}, \omega_V)  \ar[rd]^-{t_{\mathring{V}}} & \, \, \, \, \, \, \, \tn{ and} & \! \! \! \! \! \! \! \! \! \bigoplus\limits_{i=1}^r H_{x_i}^d(\omega_U) \ar[r] \ar@/_5pc/[rru]_{\sum \tn{res}_{x_i}}  &H^d_c(\mathring{U}, \omega_U)  \ar[ru]_-{t_{\mathring{U}}} &  \\
		& &   &K  & & &} 
\end{align*}
in (\ref{erstes big diag}) commute by \cite[Proposition 4.2.10]{Bey97a}.
Next, we use the map $\tau \colon \Omega^d_{B/K} \nach \Omega^d_{A/K}$ from (\ref{tau vorschrift}) to obtain an induced map 
\begin{equation}\label{tau_auf_lokalkoho}
	H_{x}^d(\tau) \colon H_{x}^d(\omega_U) \nach H_{\alpha(x)}^d(\omega_V)
\end{equation}
for every point $x \in U$, via Definition \ref{defi-ind-map-localcoho} below. %and $\mf m_x \subseteq B$ the corresponding maximal ideal. 
Recall that we use the notation $
H_{x}^d(\omega_U)=H_{\mf m_x}^d(\Omega^d_{B/K})$
etc.
\begin{definition}[Induced maps on local cohomology in general]\label{defi-ind-map-localcoho}
	Let $R$ and $S$ be rings, $\varphi \colon R \nach S$ a ring morphism, $\mf b \subseteq S$ an ideal, $M$ an $R$-module and $N$ an $S$-module. Let $\rho \colon N \nach M$ be an $R$-linear map. Then we define
$$
		H_{\mf b}^j(\rho) \colon H_{\mf b}^j(N) \nach H_{\varphi^{-1}(\mf b)}^j(M)
$$
	as follows. The inclusion $\varphi(\varphi^{-1}(\mf b))S \subseteq \mf b$ implies that $\Gamma_{\mf b}(N) \subseteq \Gamma_{\varphi(\varphi^{-1}(\mf b))S}(N)$, hence there is a natural map 
	\begin{equation}\label{first map in functoriality of localcoho}
		H_{\mf b}^j(N) \nach H_{\varphi(\varphi^{-1}(\mf b))S}^j(N).	
	\end{equation}
	Moreover,
	\begin{equation}\label{sec map in funct of localcoho}
		H_{\varphi(\varphi^{-1}(\mf b))S}^j(N) \cong H_{\varphi^{-1}(\mf b)}^j(N)
	\end{equation}
	by independence of base \cite[Proposition 2.14 (2)]{Hun07}. Finally, $\Gamma_{\varphi^{-1}(\mf b)}$ is a functor on $R$-modules and hence gives rise to 
	\begin{equation}\label{third map in funct of localcoho}
		H_{\varphi^{-1}(\mf b)}^j(N) \nach H_{\varphi^{-1}(\mf b)}^j(M).
	\end{equation}
	The desired map $H_{\mf b}^j(\rho)$ is the composite of these three maps.
\end{definition}
Choosing $R=A, S=B, \mf b=\mf m_x, M=\Omega^d_{A/K}, N=\Omega^d_{B/K}$ and $\rho=\tau$ in Definition \ref{defi-ind-map-localcoho}  yields $H_{x}^d(\tau)$. For each $i \in \{1,\ldots, s'\}$ we consider the map
$$
	\bigoplus\limits_{x \in \alpha^{-1}(y_i)} H_{x}^d(\omega_U) \xrightarrow{\sum_{x} H_{x}^d(\tau)} H_{y_i}^d(\omega_V).
$$
\begin{lemma}[Explicit description of $\sum_{x} H_{x}^d(\tau)$]\label{lem-exp-descr}
	Let $M=\Omega^d_{A/K}, N=\Omega^d_{B/K}$ and consider the map $\tau \colon N \nach M$ from  (\ref{tau vorschrift}). Let  $\mf m_y \subseteq A$ be a maximal ideal that pulls back to $\mf m$ in $T_d$, where $\mf m$ denotes the ideal corresponding to the point $0 \in \D^d$. For every $ \mf m_x \subseteq B$ that pulls back to $\mf m_y$, taking completions in the ring diagram (\ref{diag dreieck}) yields the diagram
\begin{equation*}%\label{diag dreieck compl}
	\nonumber
		\begin{gathered}
			\begin{tikzcd}
				\com{A}{\mf m_y} \arrow[d] & \com{T_d}{\mf m}, \arrow[l] \arrow[ld] \\
				\com{B}{\mf m_x}          &                                                                          
			\end{tikzcd}
		\end{gathered}
\end{equation*}
	where all morphisms 
	are finite.
%	due to \tp{Proposition} \ref{prop finiteness in completion}. 
	Let $X_1, \ldots, X_d$ be a system of parameters for $\com{T_d}{\mf m}$. Then the images in $\com{A}{\mf m_y}$ resp.\ $\com{B}{\mf m_x}$ are also a system of parameters 
%	xdue to \tp{Corollary} \ref{cor sys of par}. 
	 and the map
$$
		\bigoplus\limits_{x \in \alpha^{-1}(y)} H_{x}^d(N) \xrightarrow{\sum_{x} H_{x}^d(\tau)} H_{y}^d(M)
$$
	identifies, via %\tp{Mayer-Vietoris and} 
Lemma \ref{koszul coho}, with the map 
	\begin{equation}\label{map endgame}
		\varinjlim_\rho  \com{N}{\mf m_{y}} / (X_1^\rho, \ldots, X_n^\rho) \xrightarrow{\com{\tau}{\mf m_y}} \varinjlim_\rho  \com{M}{\mf m_{y}} / (X_1^\rho, \ldots, X_n^\rho).
	\end{equation}
More precisely: The diagram 
	\begin{equation}\label{diag completions rel trace}
		{\scalefont{0.92}
		\begin{tikzcd}%[column sep=1.1cm]
			\bigoplus\limits_{x \in \alpha^{-1}(y)} H_{\wh{\mf m_x}}^d(\wh{N}) \arrow[r, rightarrow, "\sum_x (\ref{first map in functoriality of localcoho})"', "\cong"] &  {H_{\wh{\mf m_y B}}(\com{N}{\mf m_y B})} \arrow[r, Rightarrow, no head, "(\ref{sec map in funct of localcoho})"'] &   {H_{\wh{\mf m_y}}(\com{N}{\mf m_y})} \arrow[r, "\com{\tau}{\mf m_y}", "\tblack{(\ref{third map in funct of localcoho})}"']  \arrow[d, Rightarrow, "\tn{Lemma } \ref{koszul coho}"' , no head] & {H_{\wh{\mf m_y}}(\com{M}{\mf m_y})}                        \\
			&                           & \varinjlim_\rho  \com{N}{\mf m_{y}} / (X_1^\rho, \ldots, X_n^\rho) \arrow[r, "\com{\tau}{\mf m_y}"] &  \varinjlim_\rho  \com{M}{\mf m_{y}} / (X_1^\rho, \ldots, X_n^\rho)  \arrow[u, Rightarrow, "\tn{Lemma } \ref{koszul coho}"' , no head] 
		\end{tikzcd}
	}
\end{equation}
commutes and the composite of the arrows in the top row is the map $\sum_{x} H_{x}^d(\tau) $.
\end{lemma}
\begin{proof}
The assertions about completions preserving finiteness of morphisms and systems of parameters are proved using standard facts from commutative algebra, see \cite[Proposition 1.67 and Corollary 1.68]{Mal23} for details. 	Next, the first map
 in the top row of the diagram (\ref{diag completions rel trace}) 	is an isomorphism by the same Mayer-Vietoris argument that lies behind Lemma \ref{lemma gamma new} (\ref{rm lc may vie}). 
		By going through Definition \ref{defi-ind-map-localcoho}, %and using the insensivity to passage to completion (Proposition \ref{cohinvarianzuntervervollst}), independence of base (Proposition \ref{localcoho-properties} (ii)) as well as the equalities $\wh{\mf m_y} \wh{B}= \wh{\mf m_yB}$ and $\com{N}{\mf m_y B}= \com{N}{\mf m_y}$, 
	we see that the map $\sum_{x} H_{x}^d(\tau) $ is the composite of the arrows in the top row of the diagram  (\ref{diag completions rel trace}). 
	On the other hand, applying Lemma \ref{koszul coho} to the last arrow in the top row of the diagram produces the map (\ref{map endgame}), as required. 
\end{proof}
Taking the direct sum over all $i=1, \ldots, r'$ we find the map
$$
	\bigoplus\limits_{i=1}^r H_{x_i}^d(\omega_U)=\bigoplus\limits_{i=1}^{r'} \bigoplus\limits_{x \in \alpha^{-1}(y_i)} H_{x}^d(\omega_U) \nach \bigoplus\limits_{i=1}^{r'} H_{y_i}^d(\omega_V)
$$
which we denote by $\oplus_i H_{x_i}^d(\tau)$, abusing the notation. This yields the diagram
\begin{equation}\label{fett-diag2}
	\begin{gathered}
		\xymatrix{
			& \bigoplus\limits_{i=1}^{r'} H_{y_i}^d(\omega_V) \ar@{}[rdd]|{?} \ar@/^5pc/[rrd]^{\sum \tn{res}_{y_i}} \ar[r]  &H^d_c(\mathring{V}, \omega_V)  \ar[rd]^-{t_{\mathring{V}}} \ar@{}[rdd]|{?} \\
			& &    &K \\
			&\bigoplus\limits_{i=1}^r H_{x_i}^d(\omega_U) \ar[r]^-{\textnormal{dense}} \ar@/_5pc/[rru]_{\sum \tn{res}_{x_i}} \ar^{\oplus_i H_{x_i}^d(\tau)}[uu] &H^d_c(\mathring{U}, \omega_U)  \ar[ru]_-{t_{\mathring{U}}} \ar[uu]^-{q_\alpha} &} 
	\end{gathered}
\end{equation} 
with outer semicircle
\begin{equation}\label{kreis2}
	\begin{gathered}
		\xymatrix{
			&\bigoplus\limits_{i=1}^r H_{x_i}^d(\omega_U)  \ar^{\oplus_i H_{x_i}^d(\tau)}[rr] \ar@/_1pc/_{\sum \tn{res}_{x_i}}[rd] &\ar@{}[d]|{?} & \bigoplus\limits_{i=1}^{r'} H_{y_i}^d(\omega_V) \ar@/^1pc/^{\sum \tn{res}_{y_i}}[ld] \\
			& &K &}
	\end{gathered}
\end{equation}
where question marks are again placed as a reminder that commutativity needs to be shown. \\
We now explain how the proof of Theorem \ref{simpl thm rel tr} reduces to showing the commutativity of the outer semicircle (\ref{kreis2}) and the commutativity of the square from diagram (\ref{fett-diag2}), which is then done in Lemma \ref{lem3} and Lemma \ref{lem4} below in the next subsection. 
\begin{proof}[Proof of Theorem \ref{simpl thm rel tr}]
	We consider the diagram (\ref{fett-diag2}).
	The objective is to show that the triangle with the question mark commutes. Since the lower horizontal map (call it $\eta$) has dense image, it suffices to show that 
	\begin{equation}\label{final eq to show}
		t_{\mathring{U}} \circ \eta = t_{\mathring{V}} \circ q_{\alpha} \circ \eta.
	\end{equation}
	By the commutativity of the lower ``slice'' in the diagram, we have $t_{\mathring{U}} \circ \eta = \sum \tn{res}_{x_i}$. On the other hand, $t_{\mathring{V}} \circ q_{\alpha} \circ \eta$ is equal to the composite of $\sum \res_{y_i}$ with $ \oplus_i H_{x_i}^d(\tau)$ by the commutativity of the square (Lemma \ref{lem4} below) and the commutativity of the upper ``slice''. But the composite of $\sum \res_{y_i}$ with $ \oplus_i H_{x_i}^d(\tau)$ then coincides with $\sum \tn{res}_{x_i}$ by the commutativity of the outer semicircle (Lemma \ref{lem3} below). Hence both sides of (\ref{final eq to show}) are equal to $\sum \tn{res}_{x_i}$, which completes the proof that (\ref{dreieck1}) commutes and thus the proof of Theorem \ref{simpl thm rel tr}. 
\end{proof}
\subsection{Proving the missing lemmas}
We complete the proof of Theorem \ref{simpl thm rel tr} with two lemmas:
\begin{lemma}\label{lem3}
	The outer semicircle (\ref{kreis2}) commutes.
\end{lemma}
\begin{proof}
	Using the commutative diagram from Lemma \ref{lemma gamma new} (\ref{rm lc dreieck nummer}), we can re-write the diagram (\ref{kreis2}) as
	\begin{align*}
		\xymatrix{
			& \bigoplus\limits_{i=1}^r H_{x_i}^d(\omega_U) \ar[d] \ar[rr]^-{\oplus_i H_{x_i}^d(\tau)} &  &\bigoplus\limits_{i=1}^{r'} H_{y_i}^d(\omega_V) \ar[d]   \\
			& \bigoplus\limits_{i=1}^s H_{x_i}^d(\omega_U) \ar[d]_-{\gamma_U}^-\cong  & & \bigoplus\limits_{i=1}^{s'} H_{y_i}^d(\omega_V)  \ar[d]^-{\gamma_V}_-\cong   \\
			&H_0^d({\pi_U}_* \omega_U) \ar[d]_-{H_0^d(\sigma_U)} & &H_0^d({\pi_V}_* \omega_V) \ar[d]^-{H_0^d(\sigma_V)}  \\
			&H_0^d(\omega_{\D^d}) \ar[rd]_-{\res_0} &  & H_0^d(\omega_{\D^d}) \ar[ld]^-{\res_0} \\
			& & K. & }
	\end{align*}
	To show that this diagram commutes, we expand it so that it has three parts:
	\begin{align*}
		\xymatrix{
			& \bigoplus\limits_{i=1}^r H_{x_i}^d(\omega_U) \ar@{}[rrd]|{(1)} \ar[d] \ar[rr]^-{\oplus_i H_{x_i}^d(\tau)} &  &\bigoplus\limits_{i=1}^{r'} H_{y_i}^d(\omega_V) \ar[d]   \\
			& \bigoplus\limits_{i=1}^s H_{x_i}^d(\omega_U) \ar@{}[rrd]|{(2)} \ar[d]_-{\gamma_U}^-{\cong} \ar[rr]^-{\oplus_i H_{x_i}^d(\tau)} & & \bigoplus\limits_{i=1}^{s'} H_{y_i}^d(\omega_V)  \ar[d]^-{\gamma_V}_-{\cong}   \\
			&H_0^d({\pi_U}_* \omega_U) \ar[d]_-{H_0^d(\sigma_U)} \ar[rr]_-{H_0^d({\pi_V}_*(t_\alpha))} & &H_0^d({\pi_V}_* \omega_V) \ar[d]^-{H_0^d(\sigma_V)}  \\
			&H_0^d(\omega_{\D^d}) \ar[rd]_-{\res_0} \ar@{}[rr]|{(3)} &  & H_0^d(\omega_{\D^d}) \ar[ld]^-{\res_0} \\
			& & K. & }
	\end{align*}
	Then (1) commutes for obvious reasons and (3) commutes by Lemma \ref{lemma sigma}. Thus it remains to prove that (2) commutes. Letting $M=\Omega^d_{A/K}, N=\Omega^d_{B/K}$ and $\mf m$ denote the maximal ideal corresponding to $0 \in \D^d$,  Lemma \ref{lemma gamma new} (\ref{rm lc darst}) says that the square (2) is equivalent to a diagram of the form
	\begin{equation}\label{vorletz diag}
		\begin{gathered}
			\xymatrix{
				& \varinjlim_\rho \bigoplus\limits_{i=1}^s \wh{N_{\mf m_{x_i}}} / (X_1^\rho, \ldots, X_n^\rho)   \ar[rr] & & \varinjlim_\rho \bigoplus\limits_{i=1}^{s'} \wh{M_{\mf m_{y_i}}} / (X_1^\rho, \ldots, X_n^\rho)    \\
				&\varinjlim_\rho \wh{N_{\mf m}} / (X_1^\rho, \ldots, X_n^\rho)   \ar[u]_-{\tilde{\gamma}_U^{-1}}^-{\cong} \ar[rr] & & \varinjlim_\rho \wh{M_{\mf m}} / (X_1^\rho, \ldots, X_n^\rho).  \ar[u]_-{{\tilde{\gamma}_V}^{-1}}^-{\cong}  }
		\end{gathered}	
	\end{equation} 
	Moreover, for each $y_i$ we have an isomorphism 
	$	\com{B}{\mf m_{y_i}B} \cong \oplus_{x \in \alpha^{-1}(y_i)}\wh{\loc{B}{\mf m_x}} $
	of $B$-modules, whence
$$
		\wh{N_{\mf m_{y_i}}}=N \tens{B} \com{B}{\mf m_{y_i}B}=\bigoplus\limits_{x \in \alpha^{-1}(y_i)}N \tens{B}\wh{\loc{B}{\mf m_x}}=\bigoplus\limits_{x \in \alpha^{-1}(y_i)}\wh{N_{\mf m_{x}}}.
$$
	Thus
	$$
	\varinjlim_\rho \bigoplus\limits_{i=1}^s \wh{N_{\mf m_{x_i}}} / (X_1^\rho, \ldots, X_n^\rho)  \cong \varinjlim_\rho \bigoplus\limits_{i=1}^{s'} \wh{N_{\mf m_{y_i}}} / (X_1^\rho, \ldots, X_n^\rho)
	$$
	and the diagram (\ref{vorletz diag}) becomes 
	\begin{align*}
		\xymatrix{
			& \varinjlim_\rho \bigoplus\limits_{i=1}^{s'} \wh{N_{\mf m_{y_i}}} / (X_1^\rho, \ldots, X_n^\rho)   \ar[rr]^-{\oplus_i \com{\tau}{\mf m_{y_i}}} & & \varinjlim_\rho \bigoplus\limits_{i=1}^{s'} \wh{M_{\mf m_{y_i}}} / (X_1^\rho, \ldots, X_n^\rho)    \\
			&\varinjlim_\rho \wh{N_{\mf m}} / (X_1^\rho, \ldots, X_n^\rho)   \ar[u]^-{\cong} \ar[rr]^-{\com{\tau}{\mf m}} & & \varinjlim_\rho \wh{M_{\mf m}} / (X_1^\rho, \ldots, X_n^\rho)  \ar[u]^-{\cong}  }
	\end{align*}
	with horizontal maps according to Lemma \ref{lem-exp-descr} and vertical maps according to (\ref{gammatilde vorschrift}) from Lemma \ref{lemma gamma new} (\ref{rm lc darst}). This last diagram obviously commutes, so the proof is complete.
\end{proof}
\begin{lemma}\label{lem4}
	The square
	\begin{align*}
		\xymatrix{
			& \bigoplus\limits_{i=1}^{r'} H_{y_i}^d(\omega_V)  \ar[r]  &H^d_c(\mathring{V}, \omega_V)    \\
		%	& &    \\
			&\bigoplus\limits_{i=1}^{r} H_{x_i}^d(\omega_U) \ar[r] \ar^{\oplus_i H_{x_i}^d(\tau)}[u] &H^d_c(\mathring{U}, \omega_U)  \ar[u]^-{q_\alpha} } 
	\end{align*}
	from Diagram (\ref{fett-diag2}) commutes.
\end{lemma}
\begin{proof}
	Let $\mathring {\mc V}:=\pi_V^{-1}(\Dnull^d)$ and $\mathring {\mc U}:=\pi_U^{-1}(\Dnull^d)$, then it suffices to show that the following diagram commutes:
	\begin{align*}
		\xymatrix{
			& \bigoplus\limits_{i=1}^{r'} H_{y_i}^d(\omega_V) \ar@{}[rd]|{(1)}  \ar[r] & \bigoplus\limits_{i=1}^{s'} H_{y_i}^d(\omega_V) \ar@{}[rd]|{(2)}  \ar[r]  &H^d_c(\mathring{\mc V}, \omega_V)\\
		%	& & & &   \\
			& \bigoplus\limits_{i=1}^r H_{x_i}^d(\omega_U) \ar^{\oplus_i H_{x_i}^d(\tau)}[u] \ar[r] &\bigoplus\limits_{i=1}^s H_{x_i}^d(\omega_U) \ar[r]  \ar[u] &H^d_c(\mathring{\mc U}, \omega_U).  \ar[u]^-{q_\alpha}
		 } 
	\end{align*}
Now (1) commutes for obvious reasons, so it remains to prove that (2) commutes. Applying Lemma \ref{lemma gamma new} (\ref{rm lc may vie}) to the left column of (2) and the isomorphisms $H^d_c(\mathring{\mc V}, \omega_V) \cong H^d_c(\Dnull^d, {\pi_V}_* \omega_V)$ and $H^d_c(\mathring{\mc U}, \omega_U) \cong H^d_c(\Dnull^d, {\pi_U}_* \omega_U)$ to the right column, we can re-write (2) as
	\begin{align*}
		\xymatrix{
			&H_0^d({\pi_V}_* \omega_V) \ar[r]  &  H^d_c(\Dnull^d, {\pi_V}_* \omega_V)  \\
			%& &   \\
			&H_0^d({\pi_U}_* \omega_U)  \ar[u]^-{H_0^d({\pi_V}_*(t_\alpha))} \ar[r]  & H^d_c(\Dnull^d, {\pi_U}_* \omega_U). \ar[u]_-{H_c^d({\pi_V}_*(t_\alpha))} } 
	\end{align*}
	This diagram commutes by the functoriality of the horizontal maps, completing the proof.
\end{proof}
\subsection{Some consequences} Let $\alpha \colon X \nach Y$ be a finite étale morphism of smooth $d$-dimensional Stein spaces over $K$ and let $\mc G$ be a coherent sheaf on $Y$. Let  $\xi_\alpha \colon \mc G \nach \alpha_* \alpha^* \mc G$ be the adjunction morphism   and let $(-)^\vee=\Hom_K^{cont}(-,K)$ denote the continuous dual.  An easy consequence of Theorem \ref{simpl thm rel tr} is:
\begin{proposition}\label{prop compatib}
	The diagram
	\[
	\begin{tikzcd}
		{H^{d-i}_c(X, \mc \alpha^* \mc G)^\vee} \arrow[r, "\sim"] \arrow[ddd] & {\Ext^i_X(\alpha^* \mc G, \omega_X)} \arrow[d, "f \auf \alpha_*(f)"]                                                                                                  \\
		& {\Ext^i_Y(\alpha_* \alpha^* \mc G, \alpha_* \omega_X)} \arrow[d, "{\Ext^i_Y(\alpha_* \alpha^* \mc G, t_\alpha)}"]                                                       \\
		& {\Ext^i_Y(\alpha_* \alpha^* \mc G, \omega_Y                                                                             )} \arrow[d, "{\Ext^i_Y(\xi_\alpha, \omega_Y)}"] \\
		{H^{d-i}_c(Y, \mc G)^\vee} \arrow[r, "\sim"]                        & {\Ext^i_Y(\mc G, \omega_Y)}                                                                                                                                          
	\end{tikzcd}
	\]
	commutes for all $i \geq 0$, where the horizontal isomorphisms come from the Serre duality pairing. 
\end{proposition}
\begin{proof}
	As explained in the proof of \cite[Lemma 4.2.8]{SV23}, the assertion follows from Theorem \ref{simpl thm rel tr} and the naturality of the Yoneda-Cartier pairing in the coherent sheaf together with some functoriality properties.
\end{proof}
Consider the special case $\mc G= \mc O_Y$. Then $\alpha^* \mc O_Y=\mc O_X$ and hence $\Hom_X(\alpha^* \mc O_Y, \omega_X)=\Hom_X(\mc O_X, \omega_X)=\omega_X(X)$, so the commutativity of the above diagram in particular yields:
\begin{corollary}\label{cor folg}
	The diagram
	\[
	\begin{tikzcd}
		{H^d_c(X, \mc O_X)^\vee} \arrow[r, "\sim"] \arrow[d] & \omega_X(X) \arrow[d, "t_\alpha(Y)"] \\
		{H^d_c(Y, \mc O_Y)^\vee} \arrow[r, "\sim"]           & \omega_Y(Y)                         
	\end{tikzcd}
	\]
	commutes. 
\end{corollary}
Note that the domain of $t_\alpha(Y)$ is indeed $\alpha_*\omega_X(Y)=\omega_X(X)$. 
\section{Base-changing Beyer's trace map}\label{base change section}
Throughout this section, we consider the following setting: Let $K'$ be a (not necessarily finite) complete field extension of $K$  and, for any (separated) rigid space $X$ over $K$, let
$$
X':=X \ctens{K} K'
$$
denote the base change of $X$ to $K'$ as in \cite[\S 9.3.6]{BGR84}. If $R$ is a $K$-affinoid algebra, we accordingly use the notation
$$
R':= R \ctens{K} K'.
$$
Finally, let 
$
\mc F \rightsquigarrow \mc F'
$
denote the exact \quot{pullback} functor 	from coherent sheaves on $X$ to coherent sheaves on $X'$.
	In general, 
$
H^j_{c}(X, \mc F) \ctens{K} K' \ncong H^j_{c}(X', \mc F')
$
even when $K'$ is finite over $K$, due to the fact that the left-hand side is Hausdorff whereas the right-hand side can be a non-Hausdorff space, cf.\ \cite[Remark 1.11]{Bos21}. 
However, for a special affinoid wide-open space (resp.\ a Stein space) $S$ over $K$, we discuss comparison maps
$$
H^j_{c}(S, \mc F) \ctens{K} K' \nach {H^j_{c}(S', \mc F')}^\wedge
$$
in {\S}\ref{sec bc maps} (resp.\ {\S}\ref{sec bc trace}) and prove that Beyer's trace map and Serre duality behave well with respect to them.
\subsection{Comparison maps for base change}\label{sec bc maps}
Recall that,
for any rigid space $X$ with a sheaf $\mc F$ of abelian groups and $Z \subseteq X$ a finite union of admissible affinoids, 
$H^j_{Z}(X,\mc F)$ is computed by deriving the left-exact functor 
$
\Gamma_{Z}(X, \mc F):=\ker(\Gamma(X, \mc F) \nach \Gamma(X \setminus Z, \mc F).
$
%
%
%5
Moreover,
$
H^j_c(X, \mc F)=\varinjlim_{Z}H^j_{Z}(X,\mc F)
$
		where the colimit is taken over all subspaces $Z$ of the above form. We record the following obvious fact for future reference: 
	\begin{lemma}\label{cor stein coho}
		Let $S$ be a rigid space, $Z \subseteq S$ a finite union of admissible affinoids and $\mc F$ a coherent sheaf on $S$ such that $H^j(S, \mc F)=0$ for $j \geq 1$. Then the long exact sequence \cite[Remark 1.1.2 (b)]{Bey97a} yields topological isomorphisms
		$$
		H^j_{Z}(S, \mc F) \cong H^{j-1}(S \setminus Z, \mc F) \qquad \tn{ for } j \geq 2
		$$
		and
		$$
		H^1_{Z}(S, \mc F) \cong  H^{0}(S \setminus Z, \mc F)/H^0(S, \mc F).
		$$
	\end{lemma}	
\begin{proposition}\label{bc spec wide op coho}
	Let $Z$ be a connected smooth affinoid space, $\mc F$ a coherent sheaf on $Z$. Let $\mathring W \subseteq Z$ be a special affinoid wide-open space with associated finite surjective morphism 
$
		\pi \colon Z \nach \D_K^d
$
	whose restriction
	$
	\varpi \colon \oo W \nach \Dnull_K^d
	$ 
	we denote by $\varpi$.
	Let $\varepsilon \in (0,1)$ and set 
$
		S:=\oo W,  V:=\invers{\varpi}(\D^d_K(\varepsilon))$ and $X:=S \setminus V.
$
	Then:
	\begin{enumerate}[(i)]
		\item $V'=\invers{\varpi'}(\D^d_{K'}(\varepsilon))$ and
		\begin{equation}\label{nbc difference}
			X'=S' \setminus V'.
		\end{equation} \label{first statement}
		\item 	
		There is a natural map
		\begin{equation}\label{ncoho base change 1}
			H^j_{V}(S, \mc F) \tens{K} K' \nach H^j_{V'}(S', \mc F')
		\end{equation}
		for all $j \geq 0$. 
		%		In particular, 
		%
		%
		\item Taking $\varinjlim_\varepsilon$ in (\ref{ncoho base change 1}) yields a natural map
		\begin{equation}\label{c coh map}
			H^j_{c}(S, \mc F) \tens{K} K' \nach  H^j_{c}(S', \mc F')
		\end{equation}
		which induces a map on completions
		\begin{equation}\label{c coh compl map}
			H^j_{c}(S, \mc F) \ctens{K} K' \nach  H^j_{c}(S', \mc F')^\wedge.
		\end{equation}
		\item If $K'$ is moreover finite over $K$, then all three maps (\ref{ncoho base change 1}), (\ref{c coh map}) and (\ref{c coh compl map}) are isomorphisms.
	\end{enumerate} 
\end{proposition}
\begin{proof}
	\begin{enumerate}[\normalfont(i)]
		\item Since extension of scalars is compatible with the formation of fiber products, it is in particular compatible with taking preimages under morphisms. Applying this to $\varpi$, we find that
		$$
		V'=(\invers{\varpi}(\D^d_K(\varepsilon)))'=\invers{\varpi'}(\D^d_{K'}(\varepsilon))
		$$
		and
		$$
		X'=(\invers{\varpi}(\Dnull^d_K \setminus \D^d_K(\varepsilon)))'=\invers{\varpi'}((\Dnull^d_K \setminus \D^d_K(\varepsilon))').
		%	=\invers{\pi'}(\D^n_{K'} \setminus \D^n_{K'}(\varepsilon))=\D^n_{K'}\setminus \invers{\pi'}(\D^n_{K'}(\varepsilon))
		$$
		Next, we claim that 
		\begin{equation}\label{zwischenbeh}
			(\Dnull^d_K \setminus \D^d_K(\varepsilon))'=\Dnull^d_{K'} \setminus \D^d_{K'}(\varepsilon).
		\end{equation}
		To see this, note that
		\begin{equation*}
			\Dnull^d_K \setminus \mb D^d_K(\varepsilon)=\bigcup_{i, \delta} U_{K,i, \delta}  \,\, \tn{ with } \,\, U_{K,i, \delta}:=\{x \in \Dnull^d_K \colon \varepsilon + \delta \leq \btrg{x_i} \leq 1- \delta\}
		\end{equation*}  
		where $i$ runs through $1, \ldots, n$ and $\delta$ runs through a zero sequence. Due to how the base change functor is defined, $(\Dnull^d_K \setminus \D^d_K(\varepsilon))'$ is obtained by gluing the $(U_{K,i, \delta})'$. But $(U_{K,i, \delta})'=U_{K',i, \delta}$, whence (\ref{zwischenbeh}) follows.
		Altogether, we see that
		$$
		X'=\invers{\varpi'}(\Dnull^d_{K'} \setminus \D^d_{K'}(\varepsilon))=S'\setminus \invers{\varpi'}(\D^d_{K'}(\varepsilon))=S' \setminus V',
		$$
		which proves (\ref{nbc difference}).
		\item By definition,  $H^0_{V}(S, \mc F)=\ker(H^0(S, \mc F)\nach H^0(X, \mc F))$ and  $H^0_{V'}(S', \mc F')=\ker(H^0(S', \mc F')\nach H^0(X', \mc F'))$, where the latter equality uses (\ref{nbc difference}). For any rigid space $Y$ and any coherent sheaf $\mc G$ on $Y$, we consider the natural map
		$$ 
		H^0(Y, \mc G) \tens{K} K' \nach H^0(Y, \mc G) \ctens{K} K'  = H^0(Y', \mc G')
		$$
		where the last equality is due to the explicit construction of $\mc G'$. This yields the vertical maps in the commutative diagram
		\begin{equation}\label{j gleich Null diag}
			\begin{gathered}
				\begin{tikzcd}
					H^0(S', \mc F') \arrow[r]           & H^0(X', \mc F')          \\
					H^0(S, \mc F) \tens{K} K' \arrow[u] \arrow[r] & 	H^0(X, \mc F) \tens{K} K'.\arrow[u]
				\end{tikzcd}
			\end{gathered}
		\end{equation}
		By restricting the left vertical map to the kernel of the lower horizontal map (which coincides with $H^0_{V}(S, \mc F) \tens{K} K'$ due to the flatness of $K \nach K'$) and observing that this map then lands in the kernel of the upper horizontal map, we  obtain the desired map (\ref{ncoho base change 1}) for $j=0$. For $j \geq 1$ we can apply Lemma \ref{cor stein coho} and use (\ref{nbc difference})  to see that it is equivalent to produce a natural map
		\begin{equation}\label{coho change}
			H^{j-1}(X, \mc F) \tens{K} K' \nach H^{j-1}(X', \mc F').
		\end{equation}
		To construct the map (\ref{coho change}), we imitate the proof of \citestacks{02KH}, which calls for a finite Leray covering of $X$. 
		To see that there exists a finite Leray covering of $X$, first note that $\varpi$ is a finite morphism. Indeed, since $\oo W$ is a union of connected components of $\invers{\pi}(\Dnull_K^d)$ and hence a clopen subspace, the inclusion $\oo W \nachinj \invers{\pi}(\Dnull_K^d)$  is a closed immersion and in particular a finite map, whence its composite with the finite map $\invers{\pi}(\Dnull_K^d) \nach \Dnull_K^d$ is also finite, i.e.\ \linebreak $\varpi \colon \oo W \nach \Dnull_K^d$ is finite. Now, if we let $\mf W$ be the following finite Leray cover of $\Dnull^d_K \setminus \D^d_K(\varepsilon)$ 
		\begin{equation}\label{eps cover}
\Dnull^d_K \setminus \D^d_K(\varepsilon)=\bigcup_{i=1}^d U_{i, \varepsilon}  \,\, \tn{ where } \,\, U_{i, \varepsilon}:=\{x \in \Dnull^d_K \colon \varepsilon<\btrg{x_i}\},
		\end{equation}
		then $\mf U:=\invers{\varpi}\mf W$ is a finite cover of $X$ which is a Leray cover. Indeed, the latter is due to the fact that, given a space with vanishing higher coherent cohomology, its preimage under any finite morphism also has vanishing higher cohomology.
		Next, the explicit construction of $\mc F'$ and the fact that the completed tensor product commutes with finite  products yields the following relation between \v{C}ech complexes
		\begin{equation}\label{relation of complexes}
			\check{C}^{j-1}(\mf U, \mc F) \ctens{K} K' = \check{C}^{j-1}(\mf U', \mc F').
		\end{equation}
		%		which is topological since $K'$ is finite over $K$, and which remains true after taking cohomology, 
		By precomposing with the map 
		\begin{equation}\label{complex into completed complex}
			\check{C}^{j-1}(\mf U, \mc F) \tens{K} K' \nach 					\check{C}^{j-1}(\mf U, \mc F) \ctens{K} K'
		\end{equation} 
		we obtain a natural map
		\begin{equation}\label{map of complexes}
			\check{C}^{j-1}(\mf U, \mc F) \tens{K} K' \nach \check{C}^{j-1}(\mf U', \mc F').
		\end{equation}
		The cohomology of the left-hand side in (\ref{map of complexes}) is $H^{j-1}(\mf U, \mc F) \tens{K} K'$ (because $K \nach K'$ is flat), so taking cohomology in (\ref{map of complexes}) yields the desired natural map
		\begin{equation}\label{eq of cech cohos}
			H^{j-1}(\mf U, \mc F) \tens{K} K' \nach  H^{j-1}(\mf U', \mc F').
		\end{equation}
		\item We only need to see that the $V$ (resp.\ the $V'$), for varying $0<\varepsilon<1$, form a cofinal subfamily of the family of all finite unions of admissible affinoids in $S$ (resp.\ in $S'$). But the $\Dnull_K^d(\varepsilon)$  form a cofinal subfamily for $\Dnull^d_K$  and one checks that taking preimages under any finite morphism respects such cofinality, which yields the assertion for $V$.  Moreover, the same argument over $K'$, bearing in mind that $V'=\invers{\varpi'}(\D^d_{K'}(\varepsilon))$ by (\ref{first statement}), yields the assertion for $V'$.
		\item 	 Now assume that $K'$ is finite over $K$. Obviously, it suffices to show that the map (\ref{ncoho base change 1}) is an isomorphism, or, equivalently, that (\ref{coho change}) is an isomorphism, or, equivalently, that (\ref{eq of cech cohos}) is an isomorphism. For this, we first claim that the map (\ref{complex into completed complex}) is an isomorphism. Indeed, 
		 $\check{C}^{j-1}(\mf U, \mc F)$ is a Fréchet space (which is in particular Hausdorff and complete) and the functors $K' \tens{K} (-)$ and $K' \ctens{K} (-)$ are isomorphic on complete locally convex Hausdorff spaces due to $K'$ being finite over $K$, whence the claim follows.  Hence the map (\ref{map of complexes}) is also an isomorphism, %\footnote{coinciding with the isomorphism (\ref{relation of complexes})}, 
		which, by passage to cohomology, yields that (\ref{eq of cech cohos}) is now an isomorphism. This settles the assertion for $j \geq 1$. The case $j=0$ follows from the fact that both vertical maps in the diagram (\ref{j gleich Null diag}) are now isomorphisms, again because the functors $K' \tens{K} (-)$ and $K' \ctens{K} (-)$ are isomorphic on complete locally convex Hausdorff spaces.
	\end{enumerate}
\end{proof}

\subsection{Special affinoid wide-opens under base change}\label{sec 1 base change}
In this subsection we prove preparatory results for the next subsection.
\begin{lemma}\label{ctens respect fin inj}
	If $R \nach S$ is a finite (resp.\ finite injective) morphism of affinoid $K$-algebras, then the base change $R'=K' \ctens{K} R \nach K' \ctens{K} S = S'$ is also finite (resp.\ finite injective).
\end{lemma}
\begin{proof}
	The map $R'=K' \ctens{K} R \nach K' \ctens{K} S = S'$ can be identified with the map $R' \ctens{R} R \nach R' \ctens{R} S$ by the associativity of the completed tensor product \cite[2.1.7/Proposition 7]{BGR84}. This latter map arises by viewing $R \nach S$ as a map between finite $R$-modules and applying $R' \tens{R} (-)$ to it, because the functors $R' \tens{R} (-)$ and $R' \ctens{R} (-)$ are isomorphic on finite $R$-modules by \cite[Lemma 1.1.5]{Con99}. But $R' \tens{R} (-)$ preserves finiteness (and injectivity too, since $R \nach R'$ is flat by \cite[Lemma 1.1.5]{Con99}), so we are done.
\end{proof}
	\begin{lemma}\label{lem puredim}
		Let $R$ be a reduced affinoid $K$-algebra and $d=\dim R$. The following are equivalent:
		\begin{enumerate}[\normalfont(i)]
			\item $R$ has pure dimension $d$. \label{pd pure dm}
			\item Every finite injective morphism $T_d \nachinj R$ is torsion-free. \label{pd every fin inj}
			\item There exists a finite injective morphism $T_d \nachinj R$ that is torsion-free. \label{pd exist fin inj}
		\end{enumerate}
	\end{lemma}
	\begin{proof}
		Let $\mf p_1, \ldots, \mf p_s$ be the minimal prime ideals of $R$. \\ 
		(\ref{pd pure dm})$\implies$(\ref{pd every fin inj}). Let $T_d \nachinj R$ be a finite injective morphism. We have to show that no element of $T_d$, different from zero, is a zero divisor in $R$.  Since $R$ is reduced, \citestacks{00EW} tells us that the set of zero divisors in $R$ is $\bigcup_{i=1}^s \mf p_i$, so we have to show that $\mf p_i \cap T_d=0$ for each $i$. The finite injective morphism $T_d \nachinj R$ induces a finite injective morphism $T_d/(\mf p_i \cap T_d) \nachinj R/\mf p_i$, so in particular $\dim T_d/(\mf p_i \cap T_d) = \dim R/\mf p_i$. By assumption (\ref{pd pure dm}), the latter is equal to $d$, so there exists a chain of prime ideals $0 \subsetneq \ov {\mf q}_1 \subsetneq \ldots \subsetneq \ov{\mf q}_d$ of length $d$ in $T_d/(\mf p_i \cap T_d)$ and this lifts to a chain of prime ideals $\mf p_i \cap T_d \subsetneq {\mf q}_1 \subsetneq \ldots \subsetneq {\mf q}_d$ of length $d$ in $T_d$. Thus $\mf p_i \cap T_d =0$, as otherwise we would have the chain $0 \subsetneq \mf p_i \cap T_d \subsetneq {\mf q}_1 \subsetneq \ldots \subsetneq {\mf q}_d$ of length $d+1$ in $T_d$. \\
		(\ref{pd every fin inj})$\implies$(\ref{pd exist fin inj}). By Noether normalization, there exists a finite injective morphism $T_d \nachinj R$. This morphism is torsion-free by assumption (\ref{pd every fin inj}). \\
		(\ref{pd exist fin inj})$\implies$(\ref{pd pure dm}). If we have a finite injective morphism $T_d \nachinj R$ which is torsion-free, then the set $\bigcup_{i=1}^s \mf p_i$ of zero divisors in $R$ pulls back to $0$ in $T_d$, so in particular $\mf p_i \cap T_d=0$ for each $i$. Hence the morphism $T_d \nachinj R$ induces a finite injective morphism $T_d \nachinj R/\mf p_i$, so we have $\dim R/\mf p_i=\dim T_d=d$ for all $i=1, \ldots, s$.
	\end{proof}
\begin{lemma}\label{bc separable}
	Let $Z=\Sp(R)$ be a connected smooth affinoid space and 
	$
	\varphi \colon T_d \nachinj R
	$ 
	a finite injective separable morphism. Then the morphism 
	$	
	\varphi' \colon T_d(K') \nachinj R'
	$ 
	obtained by base change to $K'$ is finite, injective and separable.
\end{lemma}
\begin{proof}
	The morphism 
	$
	\varphi' \colon T_d(K') \nachinj R'
	$
	obtained by base change to $K'$ is finite injective by Lemma \ref{ctens respect fin inj}. Since $R$ is an integral domain, $R$ has pure dimension $d$, so $R'$ also has pure dimension $d$ by \cite[Lemma 2.5]{Bos70}. Since $R'$ has pure dimension $d$ and is moreover reduced, every finite injective morphism $T_d(K') \nachinj R'$ is torsion-free by Lemma \ref{lem puredim}. In particular, $\varphi'$ is torsion-free, so it induces a morphism
	$$
	Q(T_d(K')) \nachinj Q(R').
	$$
	We need to prove that this morphism is étale. Since $Q(R)$ is a separable field extension of $Q(T_d)$ by assumption, the morphism $$Q(T_d) \nachinj Q(R)$$ is étale. Setting $T_d':=T_d(K')$ and then tensoring the above morphism with $Q(T_d')$, it follows by \citestacks{00U2} that the structure morphism
$$
		Q(T_d') \nachinj Q(T_d') \tens{Q(T_d)} Q(R)
$$
	is étale too. We will show that
	\begin{equation}\label{tens ist totring}
		Q(T_d') \tens{Q(T_d)} Q(R) \cong Q(R'),
	\end{equation}
	whence $Q(R')$ is étale over $Q(T_d')$, completing the proof. \\
	\indent It remains to prove (\ref{tens ist totring}). As in the proof of Lemma \ref{ctens respect fin inj}, we see that the map $T_d' \nachinj R'$ arises by applying $T_d' \tens{T_d} (-)$ to the map $T_d \nachinj R$, i.e.\ 
$$
		R'=T_d' \tens{T_d} R.
$$
	With the multiplicative subset $S:=T_d \setminus\{0\} \subseteq T_d$, we find that
	\begin{align*}
		\invers{S}R' &=	\invers{S}(T_d') \tens{\invers{S}T_d} \invers{S}R \\
		&=\invers{S}(T_d') \tens{Q(T_d)} Q(R) 
	\end{align*}
	where the last equality holds because $\invers{S}T_d=Q(T_d)$ by definition and  $\invers{S}R =Q(R)$ by Lemma \ref{lem quotkp}.
	On the other hand, we also wish to apply Lemma \ref{lem quotkp} to the finite ring extension $\varphi' \colon T_d' \nachinj R'$. This is possible since $\varphi'$ is torsion-free, as we have established above.
	Hence Lemma \ref{lem quotkp} tells us that $Q(R')=\invers{T}R'$ with the multiplicative set $T:=T_d' \setminus \{0\}$. But obviously $\invers{T}R'=\invers{T}(\invers{S}R')$. Altogether, we have
	\begin{align*}
		Q(R') = \invers{T}(\invers{S}R') & = Q(T_d') \tens{T_d'} \invers{S} R' \\
		& = Q(T_d') \tens{T_d'}( \invers{S}(T_d') \tens{Q(T_d)} Q(R) ) \\
		&= (Q(T_d') \tens{T_d'} \invers{S}(T_d')) \tens{Q(T_d)} Q(R) \\
		& = \invers{S}Q(T_d')  \tens{Q(T_d)} Q(R) \\
		&= Q(T_d')  \tens{Q(T_d)} Q(R),
	\end{align*}
	where the last equality holds because $\invers{S}Q(T_d')=Q(T_d')$. This proves (\ref{tens ist totring}) and completes the proof of the lemma.
\end{proof}
\begin{corollary}\label{cor phi_i}
Consider again the setting of Lemma \ref{bc separable}. Let $\mf p_1, \ldots, \mf p_s$ denote the minimal prime ideals in $R'$. Then
$
R'= \prod_{i=1}^{s} R'/\mf p_i
$
and composing $\varphi'$ with the projection $R' \nachsurj R'/\mf{p}_i$ yields a finite {injective} map
$
\varphi'_i \colon T_d(K') \nachinj R'/\mf p_i
$
for all $i$. Hence there is an induced finite field extension $Q(T_d(K')) \nachinj Q(R'/\mf p_i)$ for each $i$, and these extensions are all separable. The natural maps
\begin{equation}\label{totring ist produkt}
	Q(R') \isoauf {R'}_{\mf p_1} \times \ldots \times {R'}_{\mf p_s} \isoauf Q(R'/\mf p_1) \times \ldots \times Q(R'/\mf p_s)
\end{equation}
are isomorphisms.
\end{corollary}
\begin{proof}
	Since $Z'$ is smooth, the local rings $\mc O_{Z',z'}$ are integral domains for all $z' \in Z'$. This tells us that the irreducible components of $\Spec(R')$ are pairwise disjoint and hence they coincide with the connected components of $\Spec(R')$. In other words, 
	$
	R'= \prod_{i=1}^{s} R'/\mf p_i.
	$
	Because $R'$ is reduced, $\mf p_1 \cup \ldots \cup \mf p_s$ is the set of zero divisors in $R'$ by \citestacks{00EW}. Since $\varphi'$ is torsion-free (as we have established in the proof of Lemma \ref{bc separable}), it pulls back zero divisors to zero, so in particular it pulls back each $\mf p_i$ to zero, i.e.\ each $\varphi'_i$ is indeed injective. 
	The first map in (\ref{totring ist produkt}) is an isomorphism by \citestacks{02LX}.
The localization of a reduced ring at a minimal prime ideal is a field by \citestacks{00EU}, hence $\loc{R'}{\mf p_i}$ is a field and the natural map $\loc{R'}{\mf p_i} \isoauf Q(R'/\mf p_i)$ is an isomorphism for each $i$, so it follows that the second map in (\ref{totring ist produkt}) is an isomorphism. 
Finally, in general, if $F$ is a field and $A=A_1 \times \ldots \times A_n$ is  a finite product of $F$-algebras, then $A$ is étale over $F$ if and only if each $A_i$ is étale over $F$, see \citestacks{00U2}. Therefore, Lemma \ref{bc separable} tells us (equivalently) that
the finite field extensions 
$
Q(T_d(K')) \nachinj Q(R'/\mf p_i)
$
induced by  $\varphi_i'$ 
are separable for all $i$.	
\end{proof}
\begin{corollary}\label{bc the mor def aff wide op}
	Let $Z=\Sp(R)$ be a connected smooth affinoid space, $\mathring W \subseteq Z$ be a special affinoid wide-open space and $	\pi \colon Z \nach \D_K^d$ an associated finite surjective separable morphism.  Consider the morphism $$	\pi' \colon Z' \nach \D_{K'}^d$$ obtained by base change to $K'$. Then:
	\begin{enumerate}[(i)]
		\item  $\pi'$ is finite, surjective and separable.\label{separable} 
		\item Letting $Z_1', \ldots, Z_s'$ denote the connected components of $Z'$, the restriction 
		$\pi_i' \colon Z_i' \nach \D^d_{K'}$ of $\pi'$ to $Z_i'$ is finite, surjective and separable for each $i$. \label{parts pi_i}
		\item $\oo W' \subseteq Z'$ is a finite union of connected components of $\invers{\pi'}(\Dnull_{K'}^d)$. \label{components}
	\end{enumerate}
\end{corollary}
\begin{proof}
	Assertion (\ref{separable}) follows from Lemma \ref{bc separable} and Assertion  (\ref{parts pi_i}) follows from Corollary \ref{cor phi_i}. It remains to prove (\ref{components}), i.e.\ that $\oo W' $ embeds into $Z'$ as a finite union of connected components of $\invers{\pi'}(\Dnull_{K'}^d)$.  Setting $\oo Z:= \invers{\pi}(\Dnull_{K}^d)$ and $(Z')^{\oo{\vphantom{s}}}:= \invers{\pi'}(\Dnull_{K'}^d)$, we find that $(\oo Z)'=(Z')^{\oo{\vphantom{s}}}$ since extension of scalars is  compatible with taking preimages under morphisms.
	The space $\oo W$ is a finite union of connected components of $\oo Z$, so in particular it is clopen in $\oo Z$. Thus we can apply \cite[Lemma 3.1.1]{Con99} to the open immersion $\oo W \nachinj \oo Z$ to deduce that $\oo W' \nach (Z')^{\oo{\vphantom{s}}}$ is an open immersion too. On the other hand, since base change takes closed immersions to closed immersions, we see that $\oo W' \nach (Z')^{\oo{\vphantom{s}}}$ is also a closed immersion, so $\oo W'$ is clopen in $(Z')^{\oo{\vphantom{s}}}$. Hence $\oo W'$ is a union of connected components of $(Z')^{\oo{\vphantom{s}}}$ and the assertion follows.
\end{proof}
Due to Corollary \ref{bc the mor def aff wide op}, we can construct a map 
\begin{equation}\label{sigma_pi'}
	\sigma=\sigma_{\pi'} \colon \Omega^d_{R'/K'} \nach \Omega^d_{T_n'/K'}
\end{equation}
by invoking the same formula as in Definition \ref{def trace sigma}, but now using the trace $\Tr_{L'/E'}$ of the
finite étale ring map
$
	E':=Q(T_d') \nachinj L':=Q(R')
$
induced by
$\pi' \colon Z' \nach \D^d_{K'}$.
\begin{rmk}
	In the decomposition $L'= L_1' \times \ldots \times L_s'$ according to (\ref{totring ist produkt}), each $L_i'$ is finite separable over $E'$, so there is the  map 
	$
	\Omega^d_{L'/K'} = \bigoplus_{i=1}^s \Omega^d_{L_i'/K'}   \xrightarrow{\sum_i \sigma_i} \Omega^d_{E'/K'}
	$
	which is readily seen to coincide with $\sigma$ when restricted to $ \Omega^d_{R'/K'} $.
\end{rmk}
\begin{lemma}\label{sigmas und bc}
	The diagram 
	\[
	\begin{tikzcd}
		\Omega^d_{R/K} \ctens{K} K' \arrow[r, "(\sigma_{\pi})'"] \arrow[d, Rightarrow, no head] & 	\Omega^d_{T_d'/K'} \\
		\Omega^d_{R'/K'} \arrow[ru, "\sigma_{\pi'}"']                               &  
	\end{tikzcd}
	\]
	commutes. In other words, 
	\[
	\begin{tikzcd}
		(\pi_*\omega_{Z})' \arrow[r, "(\sigma_{\pi})'"] \arrow[d, Rightarrow, no head] & \omega_{\D^d_{K'}} \\
		\pi'_*\omega_{Z'} \arrow[ru, "\sigma_{\pi'}"']                               &  
	\end{tikzcd}
	\]
	commutes.
\end{lemma}
\begin{proof}
	We write $L=Q(R), E=Q(T_d)$ and $L'=Q(R'), E'=Q(T_d')$ for brevity. We can view $(\sigma_{\pi})'$ as the composite\footnote{One obtains the third map in the composite by taking the completion of the map \linebreak$ (L \tens{E} \Omega^d_{E/K}) \tens{K} K' \nach L \tens{E} \Omega^d_{E'/K'}$ and observing that $L \ctens{E} \Omega^d_{E'/K'} = L \tens{E} \Omega^d_{E'/K'} $ since both factors in the tensor product are finite over $E$.}
	\[
	\begin{tikzcd}
		\Omega^d_{R/K} \ctens{K} K' 	 \arrow[r] & [-1.6em]  \Omega^d_{L/K} \ctens{K} K'  \arrow[r, Rightarrow, no head] & [-1.6em]  (L \tens{E} \Omega^d_{E/K}) \ctens{K} K'  \arrow[r] & [-1.6em]  L \tens{E} \Omega^d_{E'/K'}  \arrow[r, "\Tr_{L/E}\otimes \id"] & [0.2em]   \Omega^d_{E'/K'}
	\end{tikzcd}
	\]
	and $\sigma_{\pi'}$ as the composite
	\[
	\begin{tikzcd}
		\Omega^d_{R'/K'} \arrow[r] & [-1em] \Omega^d_{L'/K'} \arrow[r, Rightarrow, no head] & [-1em] L' \tens{E'} \Omega^d_{E'/K'} \arrow[r, "\Tr_{L'/E'}\otimes \id"] & [1em]   \Omega^d_{E'/K'}
	\end{tikzcd}
	\]
	whence we need to show that the diagram
	\[
	\begin{tikzcd}
		\Omega^d_{R/K} \ctens{K} K' 	 \arrow[r] \arrow[d, Rightarrow, no head] & [-1.6em]  \Omega^d_{L/K} \ctens{K} K'  \arrow[r, Rightarrow, no head] & [-1.6em]  (L \tens{E} \Omega^d_{E/K}) \ctens{K} K'  \arrow[r] & [-1.6em]  L \tens{E} \Omega^d_{E'/K'}  \arrow[r, "\Tr_{L/E}\otimes \id"] & [0.2em]   \Omega^d_{E'/K'} \\
		\Omega^d_{R'/K'} \arrow[r] & [-1em] \Omega^d_{L'/K'} \arrow[r, Rightarrow, no head] & [-1em] L' \tens{E'} \Omega^d_{E'/K'} \arrow[r, "\Tr_{L'/E'}\otimes \id"] & [1em]   \Omega^d_{E'/K'} \arrow[ru, Rightarrow, no head]
	\end{tikzcd}
	\]
	commutes.
	Under the identification
	$$
	L'= L \tens{E} E'
	$$
	from the proof of Lemma \ref{bc separable}, we find that $\Tr_{L'/E'}=\Tr_{L/E} \otimes \id_{E'}$.
	Thus we can replace the last arrow in the lower line of the above diagram with   $L \tens{E} E' \tens{E'} \Omega^d_{E'/K'} \xrightarrow{\Tr_{L/E}\otimes \id_{E'} \otimes \id}  \Omega^d_{E'/K'} $ to obtain the diagram
	\[
	\begin{tikzcd}
		\Omega^d_{R/K} \ctens{K} K' 	 \arrow[r] \arrow[d, Rightarrow, no head] & [-1.6em]  \Omega^d_{L/K} \ctens{K} K'  \arrow[r, Rightarrow, no head] & [-1.6em]  (L \tens{E} \Omega^d_{E/K}) \ctens{K} K'  \arrow[r] & [-1.6em]  L \tens{E} \Omega^d_{E'/K'}  \arrow[r, "\Tr_{L/E}\otimes \id"] & [0.2em]   \Omega^d_{E'/K'} \\
		\Omega^d_{R'/K'} \arrow[r] & [-1em] \Omega^d_{L'/K'} \arrow[r, Rightarrow, no head] & [-1em] L \tens{E} E' \tens{E'} \Omega^d_{E'/K'} \arrow[rr, "\Tr_{L/E}\otimes \id_{E'} \otimes \id"] & &   \Omega^d_{E'/K'} \arrow[u, Rightarrow, no head]
	\end{tikzcd}
	\]
	which now obviously commutes, completing the proof.
\end{proof}
\subsection{Base change results for the trace map and for Serre duality}\label{sec bc trace}
We keep the notation of the previous subsection: 	Let $Z=\Sp(R)$ be a connected smooth affinoid space, $\mathring W \subseteq Z$ a special affinoid wide-open space and $	\pi \colon Z \nach \D_K^d$ an associated finite surjective separable morphism.  Let 	$\pi' \colon Z' \nach \D_{K'}^d$ denote the morphism obtained by base change to $K'$. 
In this setting, we can use $\sigma_{\pi'}$ from (\ref{sigma_pi'}) to construct a trace map
$$
	t_{\oo W'}=t_{\pi'} \colon H^d_c(\mathring{W}', \omega_{Z'}) \nach K'
$$
exactly as in Definition \ref{defi beyers trace map} (even though $Z'$ is not necessarily connected). Letting $Z_1', \ldots, Z_s'$ denote the connected components of $Z'$, $\pi_i'$ denote the restriction of $\pi'$ to $Z_i'$, $\oo Z_i':=\invers{\pi_i'}(\Dnull^d_{K'})$ and $\oo W_i'$ denote the union of those connected components of $\oo W'$ that are contained in $\oo Z_i'$, we find the commutative diagram
\[
\begin{tikzcd}[column sep=0.5cm]
	H^d_c(\oo W', \omega_{Z'}) \arrow[d, Rightarrow, no head] \arrow[r] & H^d_c(\oo Z', \omega_{Z'}) \arrow[r, "\sim"]  & H^d_c(\Dnull_{K'}^d, \pi'_*\omega_{Z'}) \arrow[r, "H^d_c(\sigma_{\pi'})"] \arrow[d, Rightarrow, no head] & H^d_c(\Dnull_{K'}^d, \omega_{\Dnull_{K'}^d}) \\
	\bigoplus\limits_{i=1}^sH^d_c(\oo W'_i, \omega_{Z'}) \arrow[r]                                & \bigoplus\limits_{i=1}^sH^d_c(\oo Z'_i, \omega_{Z'}) \arrow[r, "\sim"]                                & \bigoplus\limits_{i=1}^s H^d_c(\Dnull_{K'}^d, {\pi'_i}_*\omega_{Z'}) \arrow[ru, "\sum_i H^d_c(\sigma_{\pi_i'})"', bend right]                &  
\end{tikzcd}
\]
which tells us that $t_{\pi'}=\sum_i t_{\pi'_i}$ and in particular shows that $t_{\pi'}$ does not depend on $\pi'$ (since we know that each $t_{\pi_i'}$ does not depend on $\pi_i'$). \vspace*{2mm} \\ 
We are now ready to prove the main results of this section, first for special affinoid wide-opens and then for Stein spaces in general.
\begin{proposition}\label{main bc prop}
	%\begin{enumerate}[\normalfont(i)]
	%\item content \label{bc for unit disk}
	Let $Z=\Sp(R)$ be a connected smooth affinoid space of dimension $d$, $\mathring W \subseteq Z$ a special affinoid wide-open space. 
	%		and $	\pi \colon Z \nach \D_K^n$ an associated finite surjective separable morphism. 
	Then the following diagram commutes:
	\begin{equation}\label{first main bc diag}
		\begin{gathered}
			\begin{tikzcd}
				H^d_c(\oo W', \omega_{Z'}) \arrow[rd, "t_{\oo W'}"]            &   \\
				& K'. \\
				K' \tens{K} H^d_c(\oo W, \omega_{Z}) \arrow[uu] \arrow[ru, "\id_{K'} \otimes t_{\oo W}"'] &  
			\end{tikzcd}
		\end{gathered}
	\end{equation}
\end{proposition}
\begin{proof}
	%	\begin{enumerate}[\normalfont(i)]
First we consider the special case $\oo W= \Dnull_{K}^d, Z=\D_K^d$ with coordinates $X=(X_1,\ldots, X_d)$. Then there is an isomorphism $H^d_c(\Dnull^d_K, \omega_{\D^d_K}) \cong K \langle X^{-1}\rangle^\dagger \cdot \frac{dX}{X}$, and analogously over $K'$. In particular,  the diagram (\ref{first main bc diag}) in the assertion is just the diagram
		\[
		\begin{tikzcd}
			K' \langle X^{-1}\rangle^\dagger \cdot \frac{dX}{X} \arrow[rd, "(\res_{K'})"]            &   \\
			& K' \\
			K' \tens{K} (K \langle X^{-1}\rangle^\dagger \cdot \frac{dX}{X}) \arrow[uu] \arrow[ru, "\id_{K'} \otimes \res_{K}"'] &  
		\end{tikzcd}
		\]
		with maps according to (\ref{res vorschrift}). But this diagram obviously commutes. Now we consider a general special affinoid wide-open space $\mathring W \subseteq Z$ and an associated finite surjective separable morphism  $\pi \colon Z \nach \D_K^d$. We have to show that the outer contour of the diagram
		\[
		\begin{tikzcd}[column sep=0.4cm]
			{H^d_c(\oo W', \omega_{Z'})} \arrow[r]          & {H^d_c(\oo Z', \omega_{Z'})} \arrow[r, "\sim"] & {H^d_c(\Dnull_{K'}^d, \pi'_*\omega_{Z'})} \arrow[r, "\sigma_{\pi'}"] & {H^d_c(\Dnull_{K'}^d, \omega_{\Dnull_{K'}^d})} \arrow[rd, "\res_{K'}"]          &    \\
			&                                                &                                                                             &                                                                    & K' \\
			{H^d_c(\oo W, \omega_{Z})_{K'}} \arrow[r] \arrow[uu] & {H^d_c(\oo Z, \omega_{Z})_{K'}} \arrow[r, "\sim"] \arrow[uu]  & {H^d_c(\Dnull_{K}^d, {\pi}_*\omega_{Z})_{K'}} \arrow[r, "(\sigma_{\pi})'"] \arrow[uu]             & {H^d_c(\Dnull_{K}^d, \omega_{\Dnull_{K}^d})_{K'}} \arrow[ru, "(\res_{K})'"'] \arrow[uu] &   
		\end{tikzcd}
		\]
		commutes, where we have written $(-)_{K'}$ instead of $K' \tens{K} (-)$ in the lower line, for ease of notation. The left-hand square certainly commutes, since the first vertical map (from the left) is just the restriction of the second map to a direct summand. Next, we claim that the middle square commutes. To see this, we first consider any given $\varepsilon \in (0,1)$ and let $\mf W$ be the finite Leray cover of $\Dnull^d_K \setminus \D^d_K(\varepsilon)$ defined in (\ref{eps cover}).
		Then we recall from (the proof of) Proposition \ref{bc spec wide op coho} that the middle square is obtained by taking $\varinjlim_\varepsilon$ of the diagrams of Čech cohomology groups
		\[
		\begin{tikzcd}
			H^{d-1}(\invers{\pi'}\mf W', \omega_{Z'}) \arrow[r, Rightarrow, no head]           & {H}^{d-1}(\mf W', \pi'_*\omega_{Z'})           \\
			H^{d-1}(\invers{\pi}\mf W, \omega_{Z}) \tens{K} K' \arrow[r, Rightarrow, no head] \arrow[u] & {H}^{d-1}(\mf W, \pi_*\omega_{Z}) \tens{K} K' \arrow[u]
		\end{tikzcd}
		\]
		which obviously commute, whence our claim follows. 
		Similarly, the commutativity of the right-hand square above follows from the commutativity of 
		\[
		\begin{tikzcd}[column sep=1.4cm]
			\check{C}^{d-1}(\mf W', \pi'_*\omega_{Z'}) \arrow[r, "\sigma_{\pi'}"]           & 		\check{C}^{d-1}(\mf W', \omega_{\Dnull^d_{K'}})         \\
			{\check{C}}^{d-1}(\mf W, \pi_*\omega_{Z}) \tens{K} K' \arrow[u] \arrow[r, "\sigma_{\pi} \otimes \id"] \arrow[u] & {\check{C}}^{d-1}(\mf W, \omega_{\Dnull^d_{K}}) \tens{K} K' \arrow[u]
		\end{tikzcd}
		\]
		which holds by Lemma \ref{sigmas und bc}. Finally, the right-hand triangle commutes by what we have already discussed. The assertion follows. % \label{bc for spec wide op}
		%	\end{enumerate}
\end{proof}
\begin{theorem}\label{main bc result}
	Let $X$ be a connected smooth Stein space of dimension $d$. Choosing an admissible open cover $\{\mathring{W}_i\}_{i \in \N}$  of $X$ consisting of special affinoid wide-open  subsets (as in Lemma \ref{lem1}) and taking the colimit over all the maps $K' \tens{K} H^d_c(\oo W_i, \omega_{X}) \to H^d_c(\oo W_i', \omega_{X'})$ yields a map 
	\begin{equation}\label{bc map main}
		K' \tens{K} H^d_c(X, \omega_{X}) \to H^d_c(X', \omega_{X'})
	\end{equation}
	which is in fact canonical and makes the following diagram commute:
	\[
	\begin{tikzcd}
		H^d_c(X', \omega_{X'}) \arrow[rd, "t_{X'}"]            &   \\
		& K'. \\
		K' \tens{K} H^d_c(X, \omega_{X}) \arrow[uu] \arrow[ru, "\id \otimes t_{X}"'] &  
	\end{tikzcd}
	\]
%	Taking completions, we find that
%	\[
%	\begin{tikzcd}
%		H^n_c(X', \omega_{X'})^\wedge \arrow[rd, "t_{X'}"]            &   \\
%		& K' \\
%		K' \ctens{K} H^n_c(X, \omega_{X}) \arrow[uu] \arrow[ru, "(t_{X})'"'] &  
%	\end{tikzcd}
%	\]
%	commutes.
%
\end{theorem}
\begin{proof}
	Recall that the projective tensor product topology commutes with filtered colimits of locally convex vector spaces. We included a proof of this fact for direct sums in \cite[Remark 6.2 (ii)]{Mal23}, but the same reasoning applies verbatim in general. Therefore, the isomorphism $$K' \tens{K}  \varinjlim_i H^d_c(\oo W_i, \omega_{X}) \cong \varinjlim_i K' \tens{K} H^d_c(\oo W_i, \omega_{X}) $$ is topological. Now, the map (\ref{bc map main}) is by definition the unique map that makes the left-hand rectangle in 
	\[
	\begin{tikzcd}[column sep=0.4cm]
		H^d_c(X', \omega_{X'}) \arrow[rr, Rightarrow, no head]                   &                                  & \varinjlim_i H^d_c(\oo W_i', \omega_{X'}) \arrow[rd]            &    \\
		&                                  &                         & K' \\
		K' \tens{K} H^d_c(X, \omega_{X}) \arrow[uu, dotted] \arrow[r, Rightarrow, no head] & K' \tens{K} \varinjlim_i  H^d_c(\oo W_i, \omega_{X})  \arrow[r, Rightarrow, no head] &  \varinjlim_i K' \tens{K} H^d_c(\oo W_i, \omega_{X}) \arrow[uu] \arrow[ru] &   
	\end{tikzcd}
	\]
	commute, and we have to show that the resulting outer contour in the above diagram commutes (cf.\ Definition \ref{bey trace for stein def}). But this is immediate, since the right-hand triangle commutes by Proposition \ref{main bc prop}. It only remains to prove that the map (\ref{bc map main}) is canonical, i.e.\ independent of the choice of the cover $\{\mathring{W}_i\}_{i \in \N}$. Given another cover $\{\mathring{V}_i\}_{i \in \N}$ of $X$ as in the assertion, we can assume that $\oo V_i \subseteq \oo W_i$ (otherwise replace $\oo V_i$ by $\oo V_i \cap \oo W_i$ which is again special affinoid wide-open by \cite[Lemma 5.1.3]{Bey97a}). Then we need to show that the diagram
	\[
	\begin{tikzcd}
		\varinjlim_i H^d_c(\oo V_i', \omega_{X'})  \arrow[r, Rightarrow, no head] & H^d_c(X', \omega_{X'})  \arrow[r, Rightarrow, no head] & \varinjlim_i H^d_c(\oo W_i', \omega_{X'})            \\
		\varinjlim_i K' \tens{K} H^d_c(\oo V_i, \omega_{X}) \arrow[rr] \arrow[u]           &                                  & \varinjlim_i K' \tens{K} H^d_c(\oo W_i, \omega_{X}) \arrow[u]
	\end{tikzcd}
	\]
	commutes. Granting that the map  $\oo V_i' \nach \oo W_i'$ is an open immersion, we find induced maps $H^d_c(\oo V_i', \omega_{X'}) \nach H^n_c(\oo W_i', \omega_{X'})$ whose colimit over all $i$
	makes the outer contour and the upper semicircle in the extended diagram
	\[
	\begin{tikzcd}
		\varinjlim_i H^d_c(\oo V_i', \omega_{X'})  \arrow[r, Rightarrow, no head] \arrow[rr, bend left] & H^d_c(X', \omega_{X'})  \arrow[r, Rightarrow, no head] & \varinjlim_i H^d_c(\oo W_i', \omega_{X'})            \\
		\varinjlim_i K' \tens{K} H^d_c(\oo V_i, \omega_{X}) \arrow[rr] \arrow[u]           &                                  & \varinjlim_i K' \tens{K} H^d_c(\oo W_i, \omega_{X}) \arrow[u]
	\end{tikzcd}
	\]
	commute, whence the desired commutativity of the rectangle follows. It remains to check that $\oo V_i' \nach \oo W_i'$ is an open immersion. Dropping the index $i$ for notational convenience, it suffices to show that $\oo W' \nach X'$ is an open immersion (since then $\oo V' \nachinj X'$ is an open immersion too, and $\oo V'$ lands in the admissible open $\oo W' \subseteq X'$). Adopting our standard notation $\mathring W \subseteq Z$ and $	\pi \colon Z \nach \D_K^d$, the desired conclusion follows since $\oo W'$ is an admissible open in $\invers{\pi'}(\Dnull_{K'}^d)$ by Proposition \ref{bc the mor def aff wide op} (\ref{components}), $\invers{\pi'}(\Dnull_{K'}^d)$   is certainly an admissible open in $Z'$, and $Z'$ is an admissible open in $X'$ by the explicit construction of $X'$. 
\end{proof}
Next, we prove that the Yoneda-Cartier pairing is also compatible with base change:
\begin{proposition}\label{bc yoneda-cartier}
	Let $X$ be a smooth rigid Stein $K$-space of dimension $d$ and let $\mc F$ be a coherent sheaf on $X$. Then the following diagram commutes for all $i \geq 0$:
	\[
	\begin{tikzcd}
		H^{d-i}_c(X', \mc F')   \arrow[r, phantom, sloped, "\times"] & \Ext^{i}_{X'}(\mc F', \omega_{X'}) \arrow[r]  &  H^{d}_c(X', \omega_{X'})  \\
		H^{d-i}_c(X, \mc F)  \arrow[r, phantom, sloped, "\times"]    \arrow[u]      & \Ext^{i}_X(\mc F, \omega_X) \arrow[r]   \arrow[u]        & H^{d}_c(X, \omega_X).       \arrow[u]
	\end{tikzcd}
	\]
\end{proposition}
\begin{proof}
Let ${\alpha_{i,\mc F}}\colon H^{d-i}_c(X, \mc F) \nach H^{d-i}_c(X', \mc F')$ denote the comparison map for base change. 	We will prove the following equivalent formulation of the theorem:  the diagram 
	\[
	%\adjustbox{scale=0.92,center}{%
		{\scalefont{0.92}
			\begin{tikzcd}[column sep=0.3cm]
				{\Ext^i_{X'}(\mc F', \omega_{X'})} \arrow[r]         & {\Hom_{K'}(H^{d-i}_c(X', \mc F'), H^{d}_c(X', \omega_{X'})) } \arrow[rd, "(-)\circ \alpha_{i,\mc F}"]                      &    \\
				&                                                                             & \Hom_K(H^{d-i}_c(X, \mc F), H^{d}_c(X', \omega_{X'}))  \\
				{\Ext^i_X(\mc F, \omega_{X})} \arrow[uu] \arrow[r] & {\Hom_K(H^{d-i}_c(X, \mc F), H^{d}_c(X, \omega_{X}))}  \arrow[ru, " \alpha_{0,\omega_X} \circ (-)"'] &   
			\end{tikzcd}
		}
		\]
		commutes for all $i \geq 0$. 
		We can view the content of this diagram as having two maps from the $\delta$-functor $\Ext^i_X(-, \omega_{X})$ to the $\delta$-functor $\Hom_K(H^{d-i}_c(X, -), H^{d}_c(X', \omega_{X'}))$, and we want to prove that these maps coincide. Note that the latter is indeed a $\delta$-functor, since it is the composite  $\Hom_K(-, H^{d}_c(X', \omega_{X'})) \circ H^{d-i}_c(X, -)$ of the $\delta$-functor $H^{d-i}_c(X, -)$ with the exact functor $\Hom_K(-, H^{d}_c(X', \omega_{X'}))$. Now, since $\Ext^i_X(-, \omega_{X})$ is a universal $\delta$-functor, it suffices to show that the mentioned maps coincide for $i=0$, i.e.\ that the above diagram commutes for $i=0$. For this, consider any given $\gamma \in \Hom_X(\mc F, \omega_{X})$. If we denote its image in $\Hom_{X'}(\mc F', \omega_{X'})$ by $\gamma'$ -{}- so, for $U \subseteq X$ affinoid, $\gamma'$ over $U'$ is  
		$\mc{F'}(U')=\mc{F}(U) \ctens{K} K' \xrightarrow{\gamma \ctens{} \id} \omega_X(U) \ctens{K} K'=\omega_{X'}(U')$ -{}- then the commutativity of the above diagram for $i=0$ amounts to the commutativity of 
		\[
		\begin{tikzcd}[column sep=2cm]
			H^{d}_c(X', \mc F') \arrow[r, "{H^d_c(X', \gamma')}"]           & H^{d}_c(X', \omega_{X'})           \\
			H^{d}_c(X, \mc F) \arrow[u, "\alpha_{0,\mc F}"'] \arrow[r, "{H^d_c(X, \gamma)}"] & H^{d}_c(X, \omega_{X}) \arrow[u, "\alpha_{0,\omega_X}"']
		\end{tikzcd}
		\]
		which holds true for all $\gamma \in \Hom_X(\mc F, \omega_{X})$ since it is evident from the construction of our base-change-comparison maps $\alpha$ that they are functorial in this sense. Thus the proposition is proved. 
	\end{proof}
	Now we can combine the content of Theorem \ref{main bc result} and Proposition \ref{bc yoneda-cartier} into the following:
	\begin{corollary}\label{bc ult result}
		Let $X$ be a smooth rigid Stein $K$-space of dimension $d$. Then the Serre duality pairing from Theorem \ref{thm serre duality} is compatible with base change to any complete extension field $K'/K$, in the following sense: Letting $X'$
		be the base change of $X$ to $K'$  and $\mc F  \rightsquigarrow \mc F'$ be the pullback functor 
		from coherent sheaves on $X$ to coherent sheaves on $X'$, the diagram
		\[
		\begin{tikzcd}
			{H^{d-i}_c(X', \mc F')} \arrow[r, "\times", phantom]         & {\Ext^{i}_{X'}(\mc F', \omega_{X'})} \arrow[r]    & H^{d}_c(X', \omega_{X'})   \arrow[r, "t_{X'}"] & K'          \\
			{H^{d-i}_c(X, \mc F)} \arrow[r, "\times", phantom] \arrow[u] & {\Ext^{i}_X(\mc F, \omega_X)} \arrow[u] \arrow[r] & H^{d}_c(X, \omega_{X}) \arrow[u]   \arrow[r, "t_X"] & K \arrow[u]
		\end{tikzcd}
		\]
		commutes for all coherent sheaves $\mc F $ on $X$ and all $i \geq 0$.
	\end{corollary}

\providecommand{\bysame}{\leavevmode\hbox to3em{\hrulefill}\thinspace}
\providecommand{\MR}{\relax\ifhmode\unskip\space\fi MR }
% \MRhref is called by the amsart/book/proc definition of \MR.
\providecommand{\MRhref}[2]{%
	\href{http://www.ams.org/mathscinet-getitem?mr=#1}{#2}
}
\providecommand{\href}[2]{#2}

\end{document}

%% file: headpap.tex
\usepackage[toc,page]{appendix}
\usepackage[headheight=12.1pt]{geometry}
\usepackage[all,cmtip]{xy}
\usepackage{amsmath}
\usepackage{amsfonts}
\usepackage{amssymb}
\usepackage{mathrsfs}   
\usepackage{amsthm}
\usepackage{enumerate}
\usepackage{scalefnt} % For controlling the size of the font in, say, a commutative diagram
\usepackage{xcolor}
\usepackage{cite}
\usepackage{stmaryrd}
\usepackage{graphicx}
\usepackage{tikz-cd}
\usepackage{pgf}
\usepackage{eepic}
\usepackage{pstricks}
\usepackage{epsfig}
\usepackage{fancyhdr}
\usepackage{tikz,caption,float}
\usepackage[backref=page]{hyperref} % KOMMENTAR: das linktocpage macht im table of contents nur die Seitenzahlen klickbar, nicht die Namen der Abschnitte

\usepackage{amsbsy}

\numberwithin{equation}{subsection}

\let\realequation\equation
\def\equation{\setcounter{equation}{\arabic{subsubsection}}%
   \refstepcounter{subsubsection}%
   \realequation}

\usepackage{xr,mathtools}

    \definecolor{nicered}{rgb}{0.64,0,0}
        \definecolor{greeny}{rgb}{0,0.58,0}
\hypersetup{
     colorlinks   = true,
     citecolor    = nicered,
     linkcolor = blue,
     urlcolor = violet
}
    
\theoremstyle{plain}

\newtheorem{theorem}{Theorem}[section]

\newtheorem{proposition}[theorem]{Proposition}

\newtheorem{corollary}[theorem]{Corollary}

\newtheorem{lemma}[theorem]{Lemma}

\newtheorem{innercustomthm}{Theorem}
\newenvironment{customthm}[1]
{\renewcommand\theinnercustomthm{#1}\innercustomthm}
{\endinnercustomthm}
\theoremstyle{definition}
\newtheorem{definition}[theorem]{Definition}

\newtheorem{rmk}[theorem]{Remark}

\theoremstyle{remark}

\newcommand{\N}{{\mathbb N}}

\renewcommand{\lim}[1]{\underset{#1}{\mathrm{lim}}\,}

\newcommand{\D}{\mathbb{D}}
\newcommand{\Dnull}{\mathring{\mathbb{D}}}

\newcommand{\nach}{\to}
\newcommand{\auf}{\mapsto}
\newcommand{\nachinj}{\hookrightarrow}
\newcommand{\negjoinrel}{\mathrel{\mkern3mu}}
\newcommand{\rlaprel}[1]{\mathrel{\mathrlap{#1}}}
\newcommand{\nachsurj}{\rlaprel{\to}\negjoinrel\to}
\newcommand{\btrg}[1]{\lvert #1 \rvert}

\newcommand{\tn}[1]{\textnormal{#1}}
\newcommand{\mc}[1]{\mathcal{#1}}

\newcommand{\mf}[1]{\mathfrak{#1}}
\newcommand{\mb}[1]{\mathbb{#1}}

\newcommand{\tblack}[1]{\textcolor{black}{#1}}
\newcommand{\oo}[1]{\mathring{#1}}
\newcommand{\wt}[1]{\widetilde{#1}}
\newcommand{\wh}[1]{\widehat{#1}}
\newcommand{\ov}[1]{\overline{#1}}
\newcommand{\quot}[1]{``{#1}''}  %%%%%%% Shortcut for quotation marks!!!!
\newcommand{\invers}[1]{{#1}^{-1}}
\newcommand\isoauf{\xrightarrow{
		 \smash{\raisebox{-0.32ex}{\ensuremath{\scriptstyle\sim}}}}}  %%% Shortcut for isomorphism arrow!

\newcommand\com[2]{{#1}^{\wedge #2}}
\newcommand\loc[2]{{#1}_{#2}}

\newcommand{\citestacks}[1]{\cite[\href{https://stacks.math.columbia.edu/tag/#1}{Tag #1}]{Sta23}}
\newcommand{\doublecitestacks}[2]{\cite[\href{https://stacks.math.columbia.edu/tag/#1}{Tag #1}, \href{https://stacks.math.columbia.edu/tag/#2}{Tag #2}]{Sta23}}

\newcommand{\tens}[1]{                 %%%%%%% Shortcut for tensor product!!!!!!
	\mathbin{\mathop{\otimes_{#1}}
	}
}
\newcommand{\ctens}[1]{
	\mathbin{\mathop{\widehat{\otimes}_{#1}}
	}
}
\newcommand{\supr}[1]{\lvert #1 \rvert_{\ssup}}
\newcommand{\faserphi}[1]{\invers{\phi}(#1)}
\newcommand{\faserphit}[1]{\invers{\wt{\phi}}(#1)}

\DeclareMathOperator{\Hom}{Hom}

\DeclareMathOperator{\End}{End}
\DeclareMathOperator{\Sp}{Sp}
\DeclareMathOperator{\Gal}{Gal}

\DeclareMathOperator{\Ext}{Ext}

\DeclareMathOperator{\Spec}{Spec}

\DeclareMathOperator{\res}{res}

\DeclareMathOperator{\ssup}{sup}

\DeclareFontFamily{U}{mathc}{}
\DeclareFontShape{U}{mathc}{m}{it}%
{<->s*[1.03] mathc10}{}
\DeclareMathAlphabet{\mathcal}{U}{mathc}{m}{it}
\DeclareMathOperator{\sHom}{\mathcal{H\mkern-3mu om}}
\DeclareMathOperator{\sExt}{\mathcal{E\mkern-3mu xt}}

\newcommand{\id}{\mathrm{id}}

\renewcommand{\to}{\rightarrow}

\newcommand{\Tr}{\mathrm{Tr}}